\documentclass[a4paper,reqno]{amsart}

	\usepackage[utf8]{inputenc}
	\usepackage[usenames,dvipsnames]{xcolor}
	\usepackage{amssymb,amsmath,amsthm}
	\usepackage[abbrev]{amsrefs}
	\usepackage[all]{xy} 
	\usepackage{graphicx} %
	\usepackage{pinlabel}
	\usepackage{enumerate}
	\usepackage{pack/SBall}
	\usepackage{pack/HFcommands}
	\usepackage{hyperref}
	\usepackage[nameinlink,capitalise,noabbrev]{cleveref} 
		\Crefname{equation}{equation}{equations}
		\Crefname{page}{page}{pages}
	\usepackage{pack/ThmDefEtc}

	\numberwithin{equation}{section}
	\DeclareGraphicsExtensions{.pdf,.png,.jpg}
	\graphicspath{{img/}}

	\definecolor{Stefan}{RGB}{0,82,0}
	\definecolor{Marco}{RGB}{0,0,90}

\begin{document}


\title{Heegaard Floer correction terms, with a twist} 
\author{Stefan Behrens}
\author{Marco Golla}
\address{Mathematical Institute, Utrecht University\\Utrecht, Netherlands}
\address{Department of Mathematics, Uppsala University\\Uppsala, Sweden}
\date{\today}

\begin{abstract}
We use Heegaard Floer homology with twisted coefficients to define numerical invariants for arbitrary closed 3--manifolds equipped torsion spin$^c$ structures, generalising the correction terms (or $d$--invariants) defined by Ozsv\'ath and Szab\'o for integer homology 3--spheres and, more generally, for 3--manifolds with standard~$\HFoo$.
Our twisted correction terms share many properties with their untwisted analogues.
In particular, they provide restrictions on the topology of 4--manifolds bounding a given 3--manifold.
\end{abstract}

\maketitle



\section{Introduction}
	\label{ch:intro}
One of the most fascinating results in low dimensional topology is Donaldson's diagonalizability theorem for intersection forms of smooth $4$--manifolds; 
it asserts that any negative definite intersection form of a closed smooth $4$--manifold is diagonalisable over~$\Z$.
Both assumptions on the $4$--manifold, smoothness and closedness, are crucial.
On the one hand, an equally fascinating result of Freedman shows that every unimodular symmetric bilinear form appears as the intersection form of some closed \emph{topological} $4$--manifold.
On the other hand, an easy construction shows that any symmetric bilinear form is the intersection form of some smooth $4$--manifold with boundary. 
Note however that one cannot control the topology of the boundary. 
In this paper we are interested in the possible intersection forms of smooth $4$--manifolds bounding a fixed $3$--manifold.
%

\subsection{The main results}
A purely algebraic result of Elkies~\cite{Elkies_trivial_lattice_1995} shows that Donaldson's theorem can be rephrased as a family of inequalities
	$c_1^2(\fs)+b_2(X)\leq 0$
where $\fs$ runs through all \spinc structures on a closed smooth $4$--manifold~$X$.
It turns out that these inequalities admit generalisations to $4$--manifolds with boundary.
The first significant progress in this direction was made by Fr{\o}yshov~\cite{Froyshov_SW_boundary_1996} using Seiberg--Witten theory and later by Ozsv\'ath and Szab\'o~\cite{OzsvathSzabo_4mfs_gradings_2003} in the context of Heegaard Floer homology.
In this paper we will define a generalisation of the correction terms defined by Ozsv\'ath and Szab\'o, using Heegaard Floer homology with twisted coefficients: to any \spinc $3$--manifold $\Yt$ we associate a rational number $\ud\Yt$, called the \emph{twisted correction term} of $\Yt$. One of the main goals of the paper is to prove the following general result.
\begin{theorem}\label{T:intersection form bound}
Let~$(Z,\fs)$ be a smooth \spinc $4$--manifold with boundary $\Yt$, and suppose that $Z$ is negative semidefinite and $c_1(\ft)$ is torsion.
We have
	\begin{equation}\label{eq:intersection form bound}
	c_1^2(\fs)+b_2^-(Z) \,\le\, 4\dtw(Y,\ft)+2b_1(Y).
	\end{equation}
\end{theorem}
As indicated above, similar inequalities were obtained by Fr{\o}yshov~\cites{Froyshov_SW_boundary_1996,Froyshov_monopolehomology_2010} for rational homology $3$--spheres and by 
Ozsv\'ath and Szab\'o~\cite{OzsvathSzabo_4mfs_gradings_2003} for $3$--manifolds with ``standard~$\HFoo$'' (see \cref{ch:standard HFoo} below).
Our approach is very similar to the one taken by Ozsv\'ath and Szab\'o, but it turns out that the use of twisted coefficients allows us to work with arbitrary $3$--manifolds.
The proof of \cref{T:intersection form bound} occupies 
\cref{ch:HF_review,ch:twisted correction terms,ch:negative semidefinite cobordisms},
including a brief review of Heegaard Floer homology with twisted coefficients and a discussion of the twisted correction terms and their properties.

Starting with \cref{ch:intersection_forms} we return to intersection forms of smooth $4$--manifolds with boundary. 
As a sample, we mention the following result although we actually prove a stronger statement in \cref{T:finiteness theorem_almost even}.
\begin{theorem}\label{T:finiteness theorem}
For any closed, oriented $3$--manifold~$Y$ there are only finitely many isometry classes of even, semidefinite symmetric bilinear forms that can appear as intersection forms of smooth $4$--manifolds bounded by~$Y$.
\end{theorem}
Note that \cref{T:finiteness theorem} cannot hold for \emph{topological} $4$--manifolds. 
Indeed, using Freedman's result one can add arbitrary unimodular summands to the intersection form of any given $4$--manifold by connect summing with suitable closed topological $4$--manifolds.
So the finiteness in \cref{T:finiteness theorem} is an inherently smooth phenomenon.

In \cref{ch:applications} we turn to some concrete examples and give some further applications.
In particular, for a surface~$\Sigma_g$ of arbitrary genus~$g$ we compute the twisted correction terms of~$\Sigma_g\times S^1$ -- which has non-standard~$\HFoo$ for~$g\geq1$ -- and use \cref{T:intersection form bound} to deduce the following.

\begin{theorem}\label{T:fillings of T3}
Let $Z$ be a smooth $4$--manifold with boundary~$T^3$ or~$\Sigma_2\times S^1$.
If the intersection form~$Q_Z$ is negative semidefinite and even, then its non-degenerate part is either trivial or isometric to~$E_8$, and both of these occur.
\end{theorem}

Again, we actually prove a slightly stronger statement (\cref{T:fillings of T3_almost even}).


\subsection{Notation and terminology}
By default, all manifolds assumed to be smooth, compact, connected, and oriented.
The letter~$Y$ will always indicate a closed $3$--manifold.
Similarly, we reserve
	~$Z$ for $4$--manifolds with connected boundary, and
	~$W$ for cobordisms between non-empty $3$--manifolds;
and if $Y=\del Z$, we refer to $Z$ as a \emph{filling} of~$Y$.
%
%
Spin$^c$ structures on $3$--manifolds will be denoted by~$\ft$ and those on $4$--manifolds by~$\fs$.
If~$\Yt$ is the \spinc boundary of~$\Zs$, then we call the latter a \emph{\spinc filling}.
Lastly, for $3$-- or $4$--manifold with torsion-free second cohomology we write~$\ft_0$ or~$\fs_0$ for the unique \spinc structure with trivial first Chern class, provided that they exist.

\subsection*{Acknowledgements} We would like to thank Paolo Lisca, Bruno Martelli, and Andr\'as Stipsicz for their encouragement; Filippo Callegaro, Andrea Maffei, and Danny Ruberman for helpful conversations.
A special thanks goes to Adam Levine for pointing out a mistake in an earlier proof of \cref{p:standardHFoo}.
The work on this project began at the MTA Alfréd Rényi Institute of Mathematics where both authors were supported by the ERC grant LDTBud; the second author received additional support from the PRIN--MIUR research project 2010--11 
\textit{``Variet\`a reali e complesse: geometria, topologia e analisi armonica''}, the FIRB research project
\textit{``Topologia e geometria di variet\`a in bassa dimensione''}, and the Alice and Knut Wallenberg Foundation.


\section{Review of Heegaard Floer homology}
	\label{ch:HF_review}

We recall some relevant definitions and facts about Heegaard Floer homology with twisted coefficients. 
The basic references for this material are \cite{OzsvathSzabo_HF_properties_2004}*{Section~8} and \cite{JabukaMark_product_formulae_2008}.
We will pay special attention to the role of ground rings.

\subsection{Twisted coefficients}

Fix a ground ring $\bF$; usually $\bF=\Z$, $\Q$ or $\bF_p$ for some prime $p$. Let~$Y$ be a closed, oriented $3$--manifold equipped with a \spinc structure~$\ft\in\Spinc(Y)$.
The input for Heegaard Floer theory is a Heegaard diagram~$(\Sigma,\bsa,\bsb)$ with some extra decorations (see~\cites{OzsvathSzabo_HF_definition_2004,OzsvathSzabo_HF_properties_2004} for details; for instance, we will suppress the basepoint from the notation).
The output is a short exact sequence of chain complexes
	\begin{equation}\label{eq:CF complexes}
	0 \lra \CFtm\Yt_\bF \lrao{\iota} \CFtoo\Yt_\bF \lrao{\pi} \CFtp\Yt_\bF \lra 0
	\end{equation}
over the ring $\bF[U]\otimes_\bF\bF[H_2(Y)]$ whose homology groups, denoted by~$\HFto\Yt_\bF$, are an invariant of~$\Yt$ known as \emph{Heegaard Floer homology} with \emph{fully twisted coefficients} in~$\bF$.
Following~\cite{JabukaMark_product_formulae_2008} we write $R_Y=\bF[H_2(Y)]$ so that $\bF[U]\otimes_\bF\bF[H_2(Y)]$ becomes~$R_Y[U]$; we also use the common shorthand notation
	\begin{equation*}
	\Tm=U\cdot\bF[U],
	\quad
	\Too=\bF[U,U\inv],
	\eqand
	\Tp=\bF[U,U\inv]\big/U\cdot\bF[U]
	\end{equation*}
for the $\bF[U]$--modules that have become known as \emph{towers}.
We think of them as relatively $\Z$--graded such that multiplication by~$U$ has degree~$-2$ and a subscript $\mc T^\circ_d$ indicates that $U^k$ lies in grading~$d-2k$.

For any $R_Y$--module $M$ one can further define Heegaard Floer homology homology groups with coefficients in~$M$
	\begin{equation}\label{eq:HF with coefficients}
	\tHFo(Y,\ft;M)_\bF= H_*\big( \CFto\Yt_\bF\otimes_{R_Y}M \big).
	\end{equation}
The most common choices for~$M$ is the ground ring~$\bF$ itself with the trivial $R_Y$--action. 
This yields the \emph{untwisted} Heegaard Floer homology groups~$\HFo\Yt_\bF$.
In all other cases it has become customary to speak of \emph{twisted coefficients}.
Note that for~$M=R_Y$ one recovers the fully twisted homology groups~$\HFto\Yt_\bF$.
We will usually suppress the ground ring in the subscript from the notation whenever this does not cause confusion, but at times this more precise notation will be convenient.
It follows from general principles of homological algebra that \eqref{eq:CF complexes} induces a long exact sequence of $R_Y[U]$--modules
	\begin{equation}\label{eq:HF long exact sequence}
	\cdots \ra 
	\HFtm(Y,\ft;M)_\bF \lrao{\iota_*} 
	\HFtoo(Y,\ft;M)_\bF \lrao{\pi_*} 
	\HFtp(Y,\ft;M)_\bF \lrao{\delta} 
	\cdots
	\end{equation}
while \eqref{eq:HF with coefficients} gives rise to a \emph{universal coefficient spectral sequence}
	\begin{equation}\label{eq:UCSS basic}
	E^2_{*,*}=\Tor^{R_Y}_*\big( \HFto_*\Yt_\bF, M \big) \Longrightarrow \HFto(Y,\ft;M)_\bF
	\end{equation}
which highlights the universal role of~$\HFto\Yt_\bF$.
We will mostly work with the fully twisted theory.
As explained in \cite{JabukaMark_product_formulae_2008}*{Section~3}, the groups $\HFto\Yt_\bF$ carry a relative $\Z$--grading.
Moreover, if $c_1(\ft)$ is torsion, then the relative $\Z$--grading can be lifted to an absolute $\Q$--grading \cite{OzsvathSzabo_4mfs_triangles_2006}*{Section~7}.
We also recall the following result due to Ozsváth and Szabó which is of fundamental importance for our work.
\begin{theorem}[\cite{OzsvathSzabo_HF_properties_2004}*{Theorem~10.12}]\label{T:twisted HF infinity}
If~$c_1(\ft)$ is torsion, then there is a unique equivalence class of orientation systems such that
	\begin{equation*}
	\HFtoo\Yt_\bF\cong \bF[U,U\inv] =\Too
	\end{equation*}
as $R_Y[U]$--modules with a trivial $R_Y$--action on~$\Too$.
\end{theorem}

\subsection{Cobordism maps}

Now let $\Ws$ be a \spinc cobordism from~$\Yt$ to~$(Y',\ft')$.
It is well-known that for any~$R_Y$--module~$M$ there are induced \emph{cobordism maps}
	\begin{equation*}
	F_{W,\fs;M}^\circ \colon \HFto(Y,\ft;M) \lra \HFto(Y',\ft';M(W)).
	\end{equation*}
The $R_{Y'}$--module $M(W)$ used in the target is defined as follows.
We consider the $(R_Y,R_Y')$--bimodule 
	\begin{equation*}
	B_W=\bF\big[H_2(Y)_W+H_2(Y')_W\big] \subset \bF[H_2(W)]
	\end{equation*}
where $H_2(Y)_W$ denotes the image of the map~$H_2(Y)\ra H_2(W)$ induced by inclusion (and similarly for~$Y'$) and define
	\begin{equation}\label{eq:cobordism module}
	M(W)= M \otimes_{R_Y} B_W.
	\end{equation}
For example, in the fully twisted case $M=R_Y$ we have~$R_Y(W)=B_W$ and we denote the cobordism map by
	\begin{equation*}
	\Ft^\circ_{W,\fs}\colon \HFto\Yt \ra \HFto(Y',\ft';B_W).
	\end{equation*}
These cobordism maps will play an important role in the proof of \cref{T:intersection form bound}.
Notice that, in contrast to the untwisted cobordism maps, the target of~$\Ft^\circ_{W,\fs}$ depends not only on~$(Y',\ft')$ but also on the cobordism~$W$ itself.
\begin{remark}\label{R:cobordism coefficient modules}
There are slightly different definitions of~$M(W)$ in the literature.
Ours is essentially the same as in~\cite{OzsvathSzabo_4mfs_triangles_2006}*{Section~2.7} except that we work in the Poincaré dual picture (using $H_2$ instead of $H^1$).
Another difference appears in~\cite{JabukaMark_product_formulae_2008}*{Section~2.2} where the the $R_{Y'}$--module~$\overline{M}\otimes_{R_Y}B_W$ is used.
Here $\overline{M}$ stands for~$M$ with the \emph{conjugate} $R_Y$--module structure (for which $h\in H_2(Y)$ acts as~$-h$).
However, note that the conjugation only affects the $R_Y$--module structure and that we have $\overline{M}\otimes_{R_Y}B_W\cong M\otimes_{R_Y}B_W$ as $R_{Y'}$--modules.
\end{remark}

\subsection{A connected sum formula for fully twisted coefficients}
We work over a fixed ground field~$\bF$. 
%
The following is a generalisation of the connected sum formula in Heegaard Floer homology (see~\cite[Theorem 6.2]{OzsvathSzabo_HF_properties_2004}) to fully twisted coefficients.

\begin{proposition}\label{t:twistedKunneth}
Let $(Y_1,\ft_1)$ and $(Y_2,\ft_2)$ be \spinc $3$--manifolds, and let $\tCFm(Y_i, \ft_i)$ be the usual chain complex computing $\tHFm(Y_i,\ft_i)$ for $i=1,2$.
Then there is an isomorphism
\begin{equation}\label{e:twistedKunneth}
\tHFm(Y_1\#Y_2,\ft_1\#\ft_2) \cong H_*\big(\tCFm(Y_1,\ft_1) \otimes_{\bF[U]} \tCFm(Y_2,\ft_2)\big)[2]
\end{equation}
where~$[2]$ indicates a grading shift by~$2$.
\end{proposition}


In the proof we will use the shorthand notation $\bS_n$ for the $3$--manifold $\#^n S^1\times S^2$ which will also appear later on. 

\begin{proof}
%
We argue as in the proof of~\cite[Theorem 6.2]{OzsvathSzabo_HF_properties_2004} to which we refer for further details and notation.
As in the untwisted case, 
it is more convenient to study the complex $\tCF^{\le 0}$ instead of $\tCFm$. 
This explains the degree shift in~\eqref{e:twistedKunneth}: $\tCF^{\le0}$ is just $\tCFm$ with a grading shift.
The main difference between the twisted and untwisted cases lies in the definition of the twisted coefficients map
\[
\tGam: \tCF^{\le0}(Y_1,\ft_1) \otimes_{\bF[U]} \tCF^{\le0}(Y_2,\ft_2) \to \tCF^{\le0}(Y_1\#Y_2,\ft_1\#\ft_2),
\]
the analogue of the untwisted map $\Gamma$.
The map $\tGam_0$ corresponding to $\Gamma_0$ in the proof of~\cite[Theorem 6.2]{OzsvathSzabo_HF_properties_2004} is defined in the same way, as the `closest point map'.
Once we have constructed $\tGam$, the rest of the argument follows verbatim, and we refer the reader to the original proof; we therefore focus only on the construction of $\tGam$.

Choose a Heegaard diagram $(\Sigma_i,\bsa_i,\bsb_i)$ for $Y_i$, where $\Sigma_i$ has genus $g_i$.
As in the untwisted case, consider the triple Heegaard diagram
\[
(\Sigma, \bsa, \bsb, \bsc) = (\Sigma_1\#\Sigma_2, \bsa_1\cup\bsa'_2, \bsb_1\cup\bsa_2, \bsb_1'\cup\bsb_2),
\]
where the primes denote small Hamiltonian perturbations.
It is immediate to check that $(\Sigma,\bsa,\bsb)$ represents $\ltilde Y_1 := Y_1\#\bS_{g_2}$, while $(\Sigma,\bsb,\bsc)$ represents $\ltilde Y_2 := Y_2\#\bS_{g_1}$, and $(\Sigma,\bsa,\bsc)$ represents $Y_1\#Y_2$.
Let $R_i := R_{Y_i}$ for $i=1,2$. 
Using the canonical splitting $H_2(\ltilde Y_i) = H_2(Y_i) \oplus H_2(\bS_{g_{3-i}})$ we can consider $R_i$ as an $R_{\ltilde Y_i}$--module.
Following the discussion in~\cite[Section 8.2.2]{OzsvathSzabo_HF_properties_2004} we see that $(\Sigma, \bsa, \bsb, \bsc)$ induces a map
\[
f^{\le0}_{\bsa,\bsb,\bsc}: \tCF^{\le0}(\ltilde Y_1,\ft_1;R_1) \otimes_{\bF[U]} \tCF^{\le0}(\ltilde Y_2,\ft_2; R_2)
\to \tCF^{\le0}(Y_1\#Y_2,\ft_1\#\ft_2).
\]
To conclude the proof, we need to find a map $\Phi_i:\tCF^{\le0}(Y_i,\ft_i) \to \tCF^{\le0}(\ltilde Y_1,\ft_i; R_i)$.
In fact, it is an easy check that $\tCF^{\le0}(\ltilde Y_i,\ft_i; R_i) \cong \tCF^{\le0}(Y_i,\ft_i)\otimes_{\bF[U]} \CF^{\le0}(\bS_{g_{3-i}})$, and the latter factor has a canonical top degree generator $\Theta_i$, which is the same considered in the proof of~\cite[Theorem 6.2]{OzsvathSzabo_HF_properties_2004}.
This gives the desired embeddings.

\end{proof}

\subsection{Twisted surgery triangles}

Let $K$ be a knot in a $3$--manifold~$Y$. 
We write $Y_\lambda=Y_\lambda(K)$ for the $\lambda$--framed surgery on~$K$ and $W_\lambda=W_\lambda(K)$ for the corresponding surgery cobordism from~$Y$ to~$Y_\lambda$.
We also write $Y_K=Y\setminus\nu K$ for the knot exterior and let~$M_K=\bF[H_2(Y_K)]$.

\begin{lemma}\label{T:surgery module}
For each framing~$\lambda$, $M_K$ has a $\bF[H_2(Y_\lambda)]$--module structure.
\end{lemma}
\begin{proof}
Suppose first that~$K$ represents a torsion class in~$H_1(Y)$.
Then there is a canonical isomorphism $H_2(Y)\cong H_2(Y_K)$ induced by the inclusion.
Similarly, if~$b_1(Y_\lambda)=b_1(Y)$, then the same reasoning applies to~$H_2(Y_\lambda)$.
If on the other hand~$b_1(Y_\lambda)>b_1(Y)$, then we have a split injection $H_2(Y_K)\ra H_2(Y_\lambda)$.
More precisely, the choice of a rational Seifert surface for~$K$ gives rise to a splitting $H_2(Y_\lambda) = H_2(Y_K)\oplus \Z[{S}]$ so that $H_2(Y_K)\cong H_2(Y_\lambda)/[{S}]$ which yields the desired module structure.
Similarly, if~$K$ has infinite order in~$H_1(Y)$, the inclusion~$Y_K\hookrightarrow Y_\lambda$ induces canonical isomorphisms~$H_2(Y_K)\cong H_2(Y_\lambda)$ and a splitting~$H_2(Y)\cong H_2(Y_K)\oplus\Z$, the latter induced by the choice of a primitive element~$\xi\in H_2(Y)$ with $\xi\cdot[K]\neq0$.
\end{proof}

\begin{proposition}\label{p:twisted_triangle}
As above, let $(K,\lambda)$ be a framed knot in a $3$--manifold~$Y$.
Then there is an exact triangle of the form
\[
\xymatrix{
\tHFp(Y;M_K)\ar[rr]^F & & \tHFp(Y_\lambda;M_K)\ar[dl]^G\\
 & \tHFp(Y_{\lambda+\mu};M_K)\ar[ul]^H
}
\]
where the maps $F$, $G$, and $H$ are induced by surgery cobordisms.
\end{proposition}

\begin{proof}
The proof of the twisted exact triangle~\cite{OzsvathSzabo_HF_properties_2004}*{Theorem~9.21} works here with only minor modification; more precisely, one only needs observe that the proof of \cite{OzsvathSzabo_HF_properties_2004}*{Proposition~9.22} applies also with coefficients in~$M_K$.
\end{proof}

Since the map~$H_1(\del Y_K)\ra H_1(Y_K)$ has rank~1, there is an essential simple closed curve~$\lambda_0\subset \del Y_K$, well-defined up to isotopy, that has finite order in~$H_1(Y_K)$.
By a slight abuse of notation we refer to~$\lambda_0$ as the \emph{0--framing} although it might not actually be a framing in general.
If either $\lambda$, $\mu$, or~$\lambda+\mu$ agrees with~$\lambda_0$, then some care has to be taken when using \cref{p:twisted_triangle} in the context of fully twisted coefficients.
Let~$\gamma\in\{\lambda,\mu,\lambda+\mu\}$ and write~$R_\gamma$ for~$R_{Y_\gamma}=\bF[H_2(Y_\gamma)]$ where $Y_\mu=Y$.
Notice that when $\gamma$ is not the 0--framing, then~%
$M_K\cong R_\gamma$ and thus $\tHFo(Y_\gamma,\ft;M_K) = \tHFo(Y_\gamma,\ft)$.
In particular, if $\lambda_0\notin\{\lambda,\mu,\lambda+\mu\}$, then \cref{p:twisted_triangle} provides a triangle for fully twisted coefficients.
On the other hand, if $\lambda$ is the 0--framing, in which case~$K$ has to be rationally null-homologous, then only~$Y$ and~$Y_{\lambda+\mu}$ appear with fully twisted coefficients.
But we can give a fairly explicit computation of the group $\tHFoo(Y_\lambda,\ft;M_K)$ which will be useful in conjunction with~\cref{p:twisted_triangle}. 
In fact, we have a free resolution of~$M_K$ as an $R_\lambda$--module
\[
\xymatrix{
0\ar[r] & {R_\lambda}\ar[r]^{\cdot(1-[\widehat{S}])} & {R_\lambda}\ar[r] & M_K\ar[r] & 0
}
\]
showing that $\Tor^{R_\lambda}_*(M_K,\Too) = \Too\oplus\Too[1]$ whenever the $R_\lambda$--action on $\Too$ is trivial (as it is in our case), and the universal coefficient spectral sequence collapses at the second page. We have thus proved the following.

\begin{proposition}\label{p:twistedcomputation}
When $\lambda$ is the $0$--framing, $\tHFoo(Y_\lambda,\ft;M_K) = \Too\oplus\Too[1]$.
\end{proposition}

Of course, analogous considerations hold when $\lambda+\mu$ is the 0--framing and when $K$ is homologically essential in $Y$ (which is equivalent to~$\mu=\lambda_0$).

\section{Twisted correction terms}
	\label{ch:twisted correction terms}
Let~$\Yt$ be a closed $3$--manifold, equipped with a \spinc structure~$\ft$ such that~$c_1(\ft)$ is torsion. 
We will refer to~$\Yt$ as a \emph{torsion} \spinc 3\emph{--manifold}.
In this section, we work over a ground field~$\bF$.
Recall that, when~$Y$ is a rational homology sphere, the untwisted group~$\HFp\Yt$ admits a $U$--equivariant splitting of the form
	\begin{equation*}
	\HFp\Yt\cong \Tp \oplus \HFp_{\mathrm{red}}\Yt.
	\end{equation*}
The \emph{correction term}~$d\Yt$ is the degree of the element in~$\HFp\Yt$ corresponding to~$U\inv\in\Tp$.
More generally, when~$b_1(Y)>0$ there is an action of the exterior algebra~$\Lambda = \Lambda^* (H_1(Y)/{\rm Tor})$ on $\HFo\Yt$; Ozsv\'ath and Szab\'o used this action to define a similar invariant for a restricted class of $3$--manifolds, namely the ones with \emph{standard}~$\HFoo$.
Recall that $\Yt$ is said to have \emph{standard}~$\HFoo$ if~$\HFoo\Yt$ is isomorphic to~$\Lambda\otimes_\Z \Z[U,U\inv]$ as a~$\Lambda$--module.
Under this assumption, the kernel of the~$\Lambda$--action on~$\HFoo$ maps to a copy of~$\Tp$ in~$\HFp\Yt$ whose least degree is called the \emph{bottom-most correction term}~$d_b\Yt$.
It is clear that $d_b\Yt$ generalizes $d\Yt$ for rational homology spheres.
We propose another generalisation that is available for \emph{all} $3$--manifolds.

\begin{definition}[Twisted correction terms]\label{D:twisted correction terms}
Let~$\Yt$ be a torsion \spinc $3$--manifold and let $\bF$ be a field of characteristic~$p$.
We define the \emph{(homological) twisted correction term} $\dtw_p\Yt\in\Q$ as the minimal grading among all non-zero elements in the image of~$\pi_*\colon\HFtoo_*(Y,\ft)_\bF \ra\HFtp_*(Y,\ft)_\bF$.
Similarly, there is a \emph{cohomological} version $\od_p\Yt\in\Q$ defined using the map $\iota^*\colon\tHF_\infty^*(Y,\ft)_\bF\ra\tHF_-^*(Y,\ft)_\bF$ on Heegaard Floer cohomology; since multiplication by $U$ increases the degree by 2 in cohomology, $\od_p\Yt$ is the maximal grading among all non-zero elements in the image of $\iota^*$.
\end{definition}

\begin{remark}
Using the universal coefficient theorem, it is easy to show that if $\bF'$ is a field extension of $\bF$, then the two corresponding correction terms coincide, hence the correction term only depends on the characteristic. In particular, this justifies the notational choice and shows that is suffices to consider $\bF = \Q$ or $\bF_p$.
\end{remark}
\begin{remark}
It is not known whether the twisted correction terms in fact depend on~$p$.
To the best of our knowledge, there are no examples for which $\ud_0\Yt$ and~$\ud_p\Yt$ are different for some~$p>0$ and our example computations in \cref{ch:applications} give the same results for all values of~$p$.
Note that a similar situation arises in Fr{\o}yshov's work on monopole Floer homology~\cite{Froyshov_monopolehomology_2010}*{p.~569}.
%
\end{remark}
\begin{proposition}\label{T:correction terms_p-dependence}
The correction term $\ud_0\Yt$ agrees with the minimal grading among all non-$\Z$--torsion elements in the image of~$\pi_*\colon\HFtoo_*\Yt_\Z\ra\HFtp_*\Yt_\Z$. 
Furthermore, we have $\ud_0\Yt\geq \ud_p\Yt$ for every prime $p$.
\end{proposition}

\begin{proof}
The universal coefficient theorem shows that $\tHFo\Yt_\Q = \tHFo\Yt_\Z\otimes_\Z\Q$, and the first statement readily follows. The second statement follows from the universal coefficient theorem applied to the change of coefficients from $\Z$ to $\bF_p$, together with the observation that $\tHFoo\Yt_\Z$ has no $\Z$--torsion.
\end{proof}

In what follows, we will be sloppy and simply write~$\ud$ instead of~$\ud_p$ to signify that the results and computations will hold regardless of the characteristic.

If~$Y$ is a rational homology $3$--sphere, then \cref{T:correction terms_p-dependence} shows that $\ud\Yt$ agrees with the usual correction term~$d\Yt$ as defined in~\cite{OzsvathSzabo_4mfs_gradings_2003}*{Definition~4.1}.
As in that case, there is an alternative description.
The long exact sequence~\eqref{eq:HF long exact sequence} together with \cref{T:twisted HF infinity} gives rise to a (non-canonical) decomposition of $R_Y[U]$--modules
	\begin{equation*}\label{eq:plus tower splitting}
	\tHFp\Yt\cong \Tp \oplus \tHFp_{\mathrm{red}}\Yt,
	\end{equation*}
where~$\tHFp_{\mathrm{red}}\Yt$ is defined as the cokernel of~$\pi_*$.
In such a decomposition~$\dtw\Yt$ appears as the minimal grading of non-zero elements in~$\Tp$.
In favorable cases, one can compute $\tHFp\Yt$ as a graded group and read off~$\ud\Yt$ directly.
\begin{example}\label{eg:T3 and S1xS2}
By a direct computation of $\tHFp$ one can check that 
	\begin{equation*}
	\dtw(S^1\times S^2,\ft_0)=-\tfrac{1}{2}
	\eqand
	\dtw(T^3,\ft_0)=\tfrac{1}{2}.
	\end{equation*}
The computation for~$S^1\times S^2$ is easy, for~$T^3$ see~\cite{OzsvathSzabo_4mfs_gradings_2003}*{Proposition~8.5}.
Moreover, since $S^1\times S^2$ and~$T^3$ have orientation-reversing diffeomorphisms that preserve~$\ft_0$ (up to conjugation), $\od$~agrees with~$-\ud$ as we will see in \cref{T:orientation reversal and conjugation} below.
\end{example}
It turns out that many properties of the usual correction terms have analogues for~$\ud\Yt$.
In the rest of this section we describe the effects on $\ud\Yt$ of 
	conjugation of \spinc structures,
	orientation reversal, 
	and connected sums.
In \cref{ch:negative semidefinite cobordisms} we will study the behavior under negative semidefinite cobordisms.

\subsection{Conjugation and orientation reversal}
Recall that in Heegaard Floer theory one identifies \spinc structures with homology classes of nowhere vanishing vector fields.
In particular, we have an on-the-nose equality $\Spinc(Y)=\Spinc(-Y)$ where~$-Y$ denotes~$Y$ with the opposite orientation.%
	\footnote{Note however that~$c_1(Y,\ft)=-c_1(-Y,\ft)=c_1(-Y,\bar{\ft})$. So some caution is needed when working with the more common shortened notation~$c_1(\ft)$.}
Moreover, if a \spinc structure~$\ft$ is represented by a vector field~$v$, then $-v$ represents the \emph{conjugate} \spinc structure which we denote by~$\bar{\ft}$.

\begin{proposition}\label{T:orientation reversal and conjugation}
The twisted correction terms 
	of~$\Yt$ 
satisfy
	\begin{equation*}
		\ud(Y,\bar{\ft})
	=	\ud(Y,\ft)
	=	-\od(-Y,\ft)
	=	-\od(-Y,\bar{\ft}).
	\end{equation*}
In particular, if $Y$ has an orientation-reversing self-diffeomorphism that preserves~$\ft$ up to conjugation, then~$\od\Yt=-\ud\Yt$.
\end{proposition}
\begin{proof}
This follows exactly as in the proof of~\cite{OzsvathSzabo_4mfs_gradings_2003}*{Proposition~4.2} with some additional input for twisted coefficients from~\cite{JabukaMark_product_formulae_2008}*{Section~6}.
\end{proof}

It is interesting to note that the proof of \cite{OzsvathSzabo_4mfs_gradings_2003}*{Proposition~4.2} also shows that for~$b_1(Y)=0$ we have~$\ud\Yt=\od\Yt$ -- both agreeing with~$d\Yt$ which therefore satisfies~$d(-Y,\ft)=-d\Yt$.
However, according to \cref{eg:T3 and S1xS2} this argument has to fail for~$b_1(Y)>0$.
In general, there is no obvious relation between~$\ud\Yt$ and~$\od\Yt$.

\subsection{Connected sums}
	\label{ch:connected sum formula}
Next we study the behavior of the twisted correction terms under the connected sum operation.
\begin{proposition}\label{p:additive}
For torsion \spinc $3$--manifolds $(Y_1,\ft_1)$ and $(Y_2,\ft_2)$ we have
	\begin{equation*}
	\ud(Y_1\#Y_2,\ft_1\#\ft_2) = \ud(Y_1,\ft_1) + \ud(Y_2,\ft_2)
	\end{equation*}
and
	\begin{equation*}
	\od(Y_1\#Y_2,\ft_1\#\ft_2) = \od(Y_1,\ft_1) + \od(Y_2,\ft_2)
	\end{equation*}
\end{proposition}

\begin{proof}
The idea is to show that, in the connected sum theorem for~$\tHFm$, the tensor product of the two towers is mapped surjectively onto the tower in the connected sum, and this immediately proves the statement.

To see this, observe that from~\eqref{e:twistedKunneth}, the K\"unneth theorem yields a short exact sequence that splits:
	\begin{multline*}
	0\lra \big(\tHFm(Y_1,\ft_1)\otimes_{\bF[U]}\tHFm(Y_2,\ft_2)\big)[2] 
	\lrao{j} 
	\tHFm(Y_1\#Y_2,\ft_1\#\ft_2) \\
	\lra \Tor^{\bF[U]}\big(\tHFm(Y_1,\ft_1), \tHFm(Y_2,\ft_2)\big)
	\lra 0.
	\end{multline*}

Also, there is a splitting of $\bF[U]$--modules $\tHFm(Y_i,\ft_i)\cong \bF[U]x_i\oplus \tHFmred(Y_i,\ft_i)$, where each element in $\tHFmred(Y_i,\ft_i)$ is $U$--torsion; this splitting is far from being unique, but, if we insist upon $x_i$ being homogeneous, the degree of $x_i$ is well-defined, and indeed $\deg x_i = \ud(Y_i,\ft_i)-2$.

Let $x = j(x_1\otimes x_2)$, which is a homogeneous element of degree $\ud(Y_1,\ft_1)+\ud(Y_2,\ft_2)-2$.
Since the short exact sequence above splits, and since in the tensor product the only non-$U$--torsion summand is $\bF[U]x$, we deduce that there is a decomposition of $\bF[U]$--modules
\[
\tHFm(Y_1\#Y_2,\ft_1\#\ft_2) \cong \bF[U]x \oplus T,
\]
where $T$ is the $U$--torsion summand. Since the decomposition above determines the degree of $x$, we obtain the desired equality.
\end{proof}

\subsection{Manifolds with standard \texorpdfstring{$\HFoo$}{HFoo}}
	\label{ch:standard HFoo}
We now compare the twisted correction terms with the bottom most correction terms that have been studied by Ozsv\'ath and Szab\'o~\cite{OzsvathSzabo_4mfs_gradings_2003} and later by Levine and Ruberman~\cite{LevineRuberman_correctionterms_arxiv}.
\begin{proposition}\label{p:standardHFoo}
For torsion \spinc $3$--manifolds $\Yt$ with standard $\HFoo$ we have
\[\ud(Y,\ft) \le d_b(Y,\ft).\]
\end{proposition}
(It is understood that $\ud$ and $d_b$ are defined using the same coefficient field, and that the statement holds for all characteristics.)
\begin{proof}
Let $H$, $\Lambda$ and $R$ denote the group $H_2(Y)$, the exterior algebra $\Lambda^*H$ and the ring $\Z[H]$ respectively; endow $\Z$ with the trivial $R$--module structure, i.e. $\Z = R/(h-1\mid h\in H)$.
Let the graded $R$--module $\Lambda_R := R\otimes_\Z\Lambda$, endowed with the trivial differential, be the $R$--module resolution of $\Z$, and let $R$, seen as a complex supported in degree 0, be the trivial resolution for $R$ as an $R$--module. The quotient map $R\to\Z$ induces a map between the two resolutions, that is an isomorphism of their degree-0 summands.
This map, in turn, induces a map of (universal coefficient) spectral sequences from $\tHFo\Yt$ to $\HFo\Yt$.

If $\Yt$ has standard $\HFoo$, the universal coefficient spectral sequence from $\tHFoo\Yt\otimes_R\Lambda_R$ to $\HFoo\Yt$ collapses at the second page, and moreover the action of $\Lambda$ on $\HFoo\Yt$ is induced by the action of $\Lambda$ on the first page of the spectral sequence. In particular, the bottom-most tower of $\HFoo\Yt$ corresponds to the degree-0 component of $\Lambda$, and it follows that $\tHFoo\Yt$ maps onto this tower under the map of spectral sequences described above.

Summing up, we have the following commutative diagram
\[
\xymatrix{
\tHFoo\Yt\ar[d]
\ar[r] & \HFoo\Yt\ar[d]
\\
\tHFp\Yt\ar[r] & \HFp\Yt
}
\]
where the top horizontal map is an isomorphism of $\tHFoo\Yt$ onto the kernel of the $\Lambda$--action on $\HFoo\Yt$. It follows that $d_b\Yt \ge \ud\Yt$.
\end{proof}

While in general we do not expect equality of $d_b$ and $\ud$ to hold, there are families of examples where the two quantities agree; for instance, all rational homology spheres, and $0$--surgeries along knots in the 3--sphere, as the following example shows.

\begin{example}
Let us consider a knot $K$ in $S^3$; it follows from~\cite[Section 4.2]{OzsvathSzabo_4mfs_gradings_2003} that $d_b(S^3_0(K)) = d_{-1/2}(S^3_0(K)) = d(S^3_{-1}(K))-1/2$.
Let us now look at the twisted surgery exact triangle of~\cite[Theorem 9.14]{OzsvathSzabo_HF_properties_2004} associated to the framings $\infty, -1$ and $0$ of $K$:
\[
\dots \to \HFp(S^3)[t,t\inv] \stackrel{F}{\to} \HFp(S^3_{-1}(K))[t,t\inv] \stackrel{G}{\to} \tHFp(S^3_0(K))\stackrel{H}{\to} \dots
\]
It is immediate to see that the map $F$ is multiplication by $(1-t)$, that the map $G$, restricted on the tower, is the modeled on the projection $\Z[t,t\inv]\to \Z[t,t\inv]/(1-t)\cong \Z$, and that the map $H$ vanishes on the tower; moreover, the map $H$ has degree $-1/2$ in the \spinc structure with trivial Chern class.
In particular,
\[
\ud(S^3_0(K)) \ge d(S^3_{-1}(K))-1/2 = d_b(S^3_0(K)).
\]
Combined with the proposition above, this shows that $\ud(S^3_0(K)) = d_b(S^3_0(K))$.
\end{example}

\section{Negative semidefinite cobordisms}
	\label{ch:negative semidefinite cobordisms}
In this section we prove the core technical result, \cref{t:main_inequality} below, which will imply \cref{T:intersection form bound}. 
We will work over the integers, but everything goes through for~$\Q$ and~$\bF_p$ with obvious modifications.
The proof is based on the strategy used in Ozsváth and Szabó's proof of Donaldson's theorem~\cite{OzsvathSzabo_4mfs_gradings_2003}*{Section~9}.
Throughout this section $\Ws$ will be a \spinc cobordism between torsion \spinc $3$--manifolds $\Yt$ and~$(Y',\ft')$.
To obtain cleaner statements we introduce the shorthand notation 
	\begin{equation*}
	\delta\Yt = 4\ud\Yt+2b_1(Y).
	\end{equation*}
\begin{theorem}\label{t:main_inequality}
Let $\Ws$ be a negative semidefinite \spinc cobordism between torsion \spinc $3$--manifolds~$\Yt$ and~$(Y',\ft')$ such that the inclusion $Y\hookrightarrow W$ induces an injection $H_1(Y;\Q)\to H_1(W;\Q)$.
Then
\begin{align}\label{eq:main inequality}
	c_1^2(\fs) + b_2^-(W) 
		&\leq \delta(Y',\ft') - \delta\Yt \\
		&=  4\dtw(Y',\ft') - 4\dtw(Y,\ft) + 2b_1(Y') - 2b_1(Y).
\end{align}
\end{theorem}
Before going into the proof we pause to derive some consequences of \cref{t:main_inequality}.
To begin with, we show that it implies \cref{T:intersection form bound}. 
\begin{proof}[Proof of \cref{T:intersection form bound}]
Given a negative semidefinite filling~$(Z,\fs)$ of~$(Y,\ft)$ we consider the \spinc cobordism from~$(S^3,\ft_0)$ to~$\Yt$ given by $W = Z\setminus B^4$ equipped with the restriction~$\fs$. 
Since $S^3$ is simply connected, \cref{t:main_inequality} applies and the desired inequality is immediate from the fact that $\delta(S^3,\ft_0) = 0$.
\end{proof}

Another consequence of \cref{t:main_inequality} is that the twisted correction terms, like ordinary and generalised correction terms, are rational cobordism invariants.
\begin{corollary}\label{T:homology cobordism invariance}
If $\Ws$ is a rational homology cobordism between $\Yt$ and $(Y',\ft')$, then $\ud\Yt = \ud(Y',\ft')$ and $\delta\Yt=\delta(Y',\ft')$.
\end{corollary}
\begin{proof}
Both $W$ and $-W$ are negative semidefinite, and both inclusions $Y, Y' \hra W$ induce isomorphisms on rational homology by assumption. 
Hence applying~\cref{t:main_inequality} to $W$ and $-W$ we get $\ud\Yt\le \ud(Y',\ft')$ and $\ud(Y',\ft')\le\ud\Yt$.
\end{proof}

\subsection{The proof of \cref{t:main_inequality}}
	\label{ch:proof of main inequality}
As mentioned above our proof of \cref{t:main_inequality} is modeled on Ozsváth and Szabó's proof of Donaldson's theorem in \cite{OzsvathSzabo_4mfs_gradings_2003}*{Section~9}.
The strategy is to equip the cobordism with a suitable handle decomposition and to investigate the behavior of the twisted correction terms under 1--, 2--, and $3$--handle attachments.
As usual, the 1-- and $3$--handles can be treated on an essentially formal level while the 2--handles require more sophisticated arguments -- in this case establishing properties of cobordism maps on~$\tHFoo$ with suitably twisted coefficients.
We begin with the 1-- and $3$--handles.
\begin{proposition}\label{T:1 and 3 handles}
If~$W$ consists of a single 1-- or $3$--handle attachment, then
	\begin{equation*}
		\dtw(Y',\ft')-\dtw(Y,\ft) = -\tfrac{1}{2} \big (b_1(Y')-b_1(Y)\big) =
				\begin{cases}
				-\tfrac{1}{2} & \text{for 1--handles}\\
				\phantom{-}\tfrac{1}{2} & \text{for $3$--handles}
				\end{cases}
	\end{equation*}
or, equivalently, $\delta(Y',\ft')=\delta\Yt$.
\end{proposition}
\begin{proof}
In the case of a 1--handle attachment we have~$Y'\cong Y\# (S^1\times S^2)$ and the claim follows from 
	\cref{p:additive}, 
	the computation of~$\dtw(S^1\times S^2,\ft_0)=-\tfrac{1}{2}$ in \cref{eg:T3 and S1xS2}, 
	and the fact that there is a unique \spinc structure on $W$~extending~$\ft$ whose restriction to~$Y'$ is torsion.
Similarly, for $3$--handles we have~$Y\cong Y'\#(S^1\times S^2)$. 
\end{proof}
For the discussion of 2--handles we switch to a more fitting notation.\label{end of interlude}
We consider a framed knot~$(K,\lambda)$ in a $3$--manifold~$Y$ and write~$Y_\lambda=Y_\lambda(K)$ and $W_\lambda=W_\lambda(K)$ for the $3$--manifold obtained by $\lambda$--framed surgery on~$K$ and the corresponding 2--handle cobordism.
We have to discuss the cobordism maps induced by~$W_\lambda$ and it turns out that we have to distinguish two cases depending on whether $K$ has infinite order in~$H_1(Y)$ or it represents a torsion class.
We begin with the former case which requires some more subtle modifications of the standard arguments for untwisted coefficients.

We first introduce some terminology.
For any subgroup~$V\subset H_2(Y)$ we define
	\begin{equation*}
	V^\perp = \Set{x\in H_1(Y)}{\text{$x\cdot v=0$ for all $v\in V$}}\subset H_1(Y).
	\end{equation*}
Note that $V^\perp$ contains the torsion subgroup of~$H_1(Y)$ and that the intersection pairing induces a canonical identification of $H_1(Y)\big/ V^\perp$ with the dual group~$\Hom(V,\Z)$.

\begin{definition}\label{D:}
Let~$\Yt$ be a torsion \spinc $3$--manifold and let~$V$ be a direct summand of~$H_2(Y)$.
Consider the coefficient module $M_V=\Z[H_2(Y)/V]$ with the obvious $R_Y$--action.
We say that  $\Yt$ has \emph{$V$--standard~$\tHFoo$} if there is an $R_Y[U]$--linear isomorphism
	\begin{equation*}
	\tHFoo(Y,\ft;M_V) \cong \Lambda^*V \otimes_\Z \Z[U,U\inv]
	\end{equation*}
such that the action of $V^\perp\subset H_1(Y)$ is annihilating while $V^*=H_1(Y)\big/ V^\perp$ acts by contraction on~$\Lambda^*V$. 

\end{definition}
\begin{example}\label{eg:standardness}
(i) 
For $V=H_2(Y)$ the above definition agrees with the usual notion of ``standard~$\HFoo$'' discussed in \cref{ch:standard HFoo}.

(ii) 
By \cref{T:twisted HF infinity} all $3$--manifolds have standard~$\tHFoo$ for $V=0$ and according to \cref{p:twistedcomputation} the same holds for any~$V$ of rank~1.

(iii)
This example will be particularly relevant and has, in fact, already appeared in the proof of \cref{t:twistedKunneth}.
Let~$\Yt$ be a \spinc \mbox{$3$--manifold};
the proof of \cite{OzsvathSzabo_HF_properties_2004}*{Proposition~6.4} shows that for any \mbox{$R_Y$--module~$M$} we have 
	\begin{equation}\label{eq:connected sum with S1xS2}
	\tHFo(Y\#\mathbb{S}_n,\ft\#\ft_0;M) \cong \tHFo(Y,\ft;M)\otimes \Lambda^*H_2(\mathbb{S}_n) 
	\end{equation}
where~$M$ is considered as and module over $R_{Y\#\mathbb{S}_n}=R_Y\otimes_\Z R_{\mathbb{S}_n}$ with trivial $R_{\mathbb{S}_n}$~action.
Moreover, the action of~$H_1(Y\#\mathbb{S}_n)$ on the right-hand  side is induced by the usual action of~$H_1(Y)$ on the first factor, and by the contraction with elements of~$H_1(\mathbb{S}_n)$ via the intersection product on the second factor.
In particular, it is a matter of checking the definition to see that~$(Y\#\mathbb{S}_n,\ft\#\ft_0)$ has standard~$\tHFoo$ \wrt the subgroup of $H_2(Y\#\mathbb{S}_n)$ corresponding to $H_2(\mathbb{S}_n)$.
\end{example}
%
%
%
%
%
\begin{proposition}\label{T:2-handles_essential}\label{T:analogue of Proposition 9.3}
Let~$\Yt$ be a torsion \spinc $3$--manifold and let~$(K,\lambda)$ be a framed knot in~$Y$ such that $K$ has infinite order in~$H_1(Y)$.
Let~$V$ be a direct summand of~$H_2(Y)$ such that $\Yt$ has $V$--standard~$\tHFoo$ and some~$v\in V$ satisfies~$[K]\cdot v\neq 0$.
Then there is a subgroup $V_K$ of~$H_2(Y_\lambda)$ such that~$M_{V_K}\cong M_V$.
Moreover, $Y_\lambda$~has $V_K$--standard~$\tHFoo$ for any torsion \spinc structure~$\ft'$ which is cobordant to~$\ft$ via~$(W_\lambda,\fs)$; and the cobordism map induces an isomorphism
	\begin{equation*}
	\tHFoo(Y,\ft;M_V)\big/\ker[K]
	\lrao{\cong}
	\tHFoo(Y_\lambda,\ft';M_V)
	\end{equation*}
where~$\ker[K]$ is the kernel of the action of~$[K]$.
\end{proposition}
\begin{proof}
One readily checks that the inclusion of~$Y$ in~$W_\lambda$ induces an isomorphism $H_2(W_\lambda)\cong H_2(Y)$.
According to \cref{eq:cobordism module} we get maps
	\begin{equation*}
	F^\infty_{W_\lambda,\fs} \colon \tHFoo(Y,\ft;M_V) \ra \tHFoo(Y_\lambda,\ft';M_V)
	\end{equation*}
for any $\fs\in \Spinc(W_\lambda)$ and $\ft'=\fs|_{Y_\lambda}$.

Arguing as in the proof of \cref{T:surgery module} we see that
	$H_2(Y_\lambda)\cong H_2(Y_K)$ and
	$H_2(Y)\cong H_2(Y_K)\oplus \Z$
where the second summand is generated by a primitive element of~$H_2(Y)$ that has non-trivial intersection with~$[K]$.
By assumption we can find such an element in~$V$.
Under the above identifications we can consider~$V_K=\Set{v\in V}{[K]\cdot v=0}$ as a subgroup of~$H_2(Y_K)$ and therefore of~$H_2(Y_\lambda)$.
Moreover, we have $H_2(Y)/V \cong H_2(Y_\lambda)/V_K$ and thus $M_V\cong M_{V_K}$.
Now suppose that $\ft'$ is torsion.
We can put the maps induced by $W_\lambda$ into a surgery triangle as before and argue as in the proof of~\cite{OzsvathSzabo_4mfs_gradings_2003}*{Proposition~9.3} that~$F^\infty_{W_\lambda,\fs}$ vanishes on~$\ker[K]$ and is injective on the quotient for all field coefficients.
The only missing piece is a bound on the rank of~$\tHFoo(Y_\lambda,\ft';M_V)$ in each degree. 
To that end, we observe that the $E_2$--term of the relevant universal coefficient spectral sequence is given by 
	\begin{align*}
	\Tor^{R_Y}\big(\tHFoo(Y_\lambda,\ft'), M_V\big)
	&\cong \Tor^{\Z[H_2(Y_\lambda)]}\big(\Z, \Z[H_2(Y)/V] \big) \otimes_\Z \Z[U,U\inv] \\
	&\cong \Tor^{\Z[V]}(\Z,\Z) \otimes_\Z \Z[U,U\inv] \\
	&\cong \Lambda^* V \otimes_\Z \Z[U,U\inv]
	\end{align*}
where the second isomorphism follows from 
	Shapiro's lemma (see~\cite{Brown_cohomology_of_groups_1982}*{p.~73}, for example).
The resulting rank bound can be used as a replacement of \cite{OzsvathSzabo_4mfs_gradings_2003}*{Lemma~9.2} in the proof of \cite{OzsvathSzabo_4mfs_gradings_2003}*{Proposition~9.3}.
\end{proof}
\begin{remark}\label{R:}
Note that in the above proof it was crucial for the action of~$K$ to have non-trivial image. 
Since any non-torsion element of~$H_1(Y)$ annihilates~$\tHFo\Yt$ (see~\cite{JabukaMark_product_formulae_2008}*{Remark~5.2}), the proof does not work for~$V=0$, that is, we cannot start with fully twisted coefficients for~$Y$.
\end{remark}

We now turn to the case when $K$ has finite order in~$H_1(Y)$. 
\begin{proposition}\label{T:2-handles_torsion}
Let~$\Yt$ and $(K,\lambda)$ be as above and suppose that $K$ has finite order in~$H_1(K)$.
If $b_2^+(W_\lambda)=0$, then $W_\lambda$ induces an isomorphism
	\begin{equation*}
	\tHFoo(Y,\ft)
	\lrao{\cong}
	\tHFoo(Y_\lambda,\ft')
	\end{equation*}
where~$\ft'$ is the restriction of an extension of~$\ft$ to the surgery cobordism.
\end{proposition}
\begin{remark}\label{R:rational linking}
For those familiar with rational linking numbers we note that the $b_2^+$--condition is equivalent to~$\mathrm{lk}_\Q(K,\lambda)\leq0$ so that the assumptions in the above propositions can be rephrased purely in $3$--dimensional terms.
\end{remark}
\begin{proof}[Proof of \cref{T:2-handles_torsion}]
The main idea is to study exact triangles relating suitable twisted Heegaard Floer homology groups of the manifolds $Y$, $Y_\lambda$, and~$Y_{\lambda+\mu}$ where the latter is obtained by $\lambda+\mu$--framed surgery on~$K$.
There are three cases to consider according to the change of~$b_1$ under the surgeries:
\begin{enumerate}[(1)]
	\item $b_1(Y) = b_1(Y_\lambda) = b_1(Y_{\lambda+\mu})$
	\item $b_1(Y) = b_1(Y_\lambda) < b_1(Y_{\lambda+\mu})$
	\item $b_1(Y) = b_1(Y_{\lambda+\mu}) < b_1(Y_\lambda)$
\end{enumerate}

Case (1) is an immediate adaptation of the proof of~\cite[Proposition 9.4]{OzsvathSzabo_4mfs_gradings_2003}. In fact, all relevant cobordisms induce maps between the fully twisted Floer homology groups, and the proof proceeds exactly as in the untwisted case.

Case (2) also follows from an adaptation of the same proof, but with more substantial modifications. In this case, in fact, there is a surgery exact triangle that reads as follows (see~\cite[Theorem 9.1]{JabukaMark_product_formulae_2008}):
\[
\xymatrix{
\tHFp(Y)[t,t\inv]\ar[rr]^F & & \tHFp(Y_\lambda)[t,t\inv]\ar[dl]^G\\
 & \tHFp(Y_{\lambda+\mu})\ar[ul]^H
}
\]
Here, $F$ is $t$--equivariant and is, in fact, the map $\underline{F}\otimes \mathbf{1}$, where $\underline{F}$ is the map induced by the surgery cobordism between the twisted Floer homology groups. Moreover, $t$ acts as the class of the capped-off surface $T\in H_2(Y_{\lambda+\mu})$. Since $T$ acts as the identity on $\tHFoo(Y_{\lambda+\mu})$, for all sufficiently large degrees the map $F$ is multiplication by $(1-t)$, and in particular it induces a surjection on the towers in $\tHFp\Yt$ for each torsion \spinc structure $\ft$ on $Y$. Now the argument runs as in the untwisted case to show the desired inequality; compare with \cite{OzsvathSzabo_HF_properties_2004}*{Theorem~9.1}.

In case (3), we use the surgery triangle of \cref{p:twisted_triangle}:
\[
\xymatrix{
\tHFp(Y)\ar[rr]^F & & \HFp(Y_\lambda;M_K)\ar[dl]^G\\
 & \tHFp(Y_{\lambda+\mu})\ar[ul]^H
}
\]
As in the proof of \cref{T:2-handles_essential} we show that the infinity version of~$G$ has the same kernel as the action of the dual knot of~$K$, say~$K'\subset Y_\lambda$.
Moreover, the usual argument shows that the infinity version of~$F$, which is just~$F^\infty_{W_\lambda,\fs}$, is injective; and by exactness it injects into~$\ker[K']$ which, according to \cref{p:twistedcomputation}, is graded isomorphic to~$\tHFoo(Y_\lambda,\ft')$.
Again observing that the argument goes through with arbitrary field coefficients, we see that $F^\infty_{W_\lambda,\fs}$ maps isomorphically onto~$\ker[K']$.
\end{proof}

\begin{proof}[Proof of \cref{t:main_inequality}]
The key is the standard observation that whenever we have a cobordism~$\Ws$ between torsion \spinc 3--manifolds~$\Yt$ and~$(Y',\ft')$ such that~$F^\infty_{W,\fs}$ is an isomorphism between fully twisted coefficients, then $\ud\Yt +\deg F^+_{W,\fs}\le \ud(Y',\ft')$, as an easy diagram chase shows.
Unfortunately, we cannot apply this argument directly because in general the target of the cobordism maps will not have fully twisted coefficients.

To circumvent this problem, we observe that the left-hand side of the inequality~\eqref{eq:main inequality} is additive while the  right-hand side behaves telescopically when two negative semidefinite cobordisms are composed.
Conversely, one can also show that the left-hand  side splits appropriately when~$W$ is cut along a separating \mbox{$3$--manifold} in its interior.
It would therefore be enough to prove \cref{t:main_inequality} for 
cobordisms consisting of single handle attachments.
In fact, this strategy works quite well since \cref{T:1 and 3 handles} covers 1-- and 3--handles, while \cref{T:2-handles_torsion} allow us to run the standard argument mentioned above.
What remains are 2--handle attachments along knots in essential homology classes.
It turns out that these actually cannot be treated separately but have to be paired with 1--handles.
It is at this point that the assumption on the map $H_1(Y;\Q)\ra H_1(W;\Q)$ becomes relevant and we are forced with the coefficient systems used in \cref{T:2-handles_essential}.

As a last preparatory remark, we can restrict our attention to the case when $H_1(Y;\Q)\ra H_1(W;\Q)$ is not only injective but actually an isomorphism.
Indeed, if it is not surjective, say it has corank~$k$, then we perform surgery on an embedded circle~$C\subset W\setminus\del W$ which represents a non-zero class in~$H_1(W;\Q)$ not contained in the image of~$H_1(Y;\Q)$.
The resulting cobordism~$W'$ has the same boundary as~$W$ and is easily seen to satisfy~$b_2^\pm(W')=b_2^\pm(W)$ and $H_1(Y;\Q)\ra H_1(W;\Q)$ has corank~$k-1$. 
Moreover, the restriction of~$\fs$ to~$W\setminus\nu C$ extends to~$W'$ and any such extension~$\fs'$ satisfies $c_1^2(\fs')=c_1^2(\fs)$.
In particular, the left-hand  side of~\eqref{eq:main inequality} is the same for~$\Ws$ and~$(W',\fs')$.
By successive surgeries we can therefore cut down~$H_1(W;\Q)$ to the image of~$H_1(Y;\Q)$.

We now begin the actual proof.
We choose a handle decomposition of~$W$ and put it in \emph{standard ordering} as defined by Ozsváth and Szabó (see~\cite{OzsvathSzabo_4mfs_gradings_2003}*{p.243}).
This means that 
	the handles are attached in order of increasing index and, moreover, 
	the 2--handle attachments are ordered such that $b_1$ of the intermediate $3$--manifolds first decreases, then stays constant, and finally increases.
For the existence of such a handle decomposition, see \cite{OzsvathSzabo_4mfs_gradings_2003}*{p.244}.
We cut $W$ into two pieces $W_{12}\cup_{N}W_{23}$ such that~$W_{12}$ contains all 1--handles and the decreasing 2--handles while $W_{23}$ contains the remaining 2-- and $3$--handles.
Observe that the $b_1$--decreasing 2--handles are exactly those that are attached along essential knots. 
So by the above remarks \cref{t:main_inequality} holds for~$W_{23}$ and we can restrict our attention to~$W_{12}$.
Since we are assuming that $H_1(Y;\Q)\ra H_1(W;\Q)$ is an isomorphism, there must be exactly as many $b_1$--decreasing 2--handles as there are 1--handles, say we have~$n$ each. 
Our goal is to show that~$(W_{12},\fs)$ induces an isomorphism between $\tHFoo\Yt$ and~$\tHFoo(N,\ft_N)$ where~$\ft_N=\fs|_N$ is easily seen to be torsion.
We further decompose~$W_{12}$ into pieces~$V_1$ and~$V_2$ along~$Y\#\mathbb{S}_n$ where~$V_i$ contains all $i$--handles.
Note that the attaching circles 
	$K_1,\dots,K_n\subset Y\#\mathbb{S}_n$
of the 2--handles span the subspace $H_1(\mathbb{S}_n;\Q)\subset H_1(Y\#\mathbb{S}_n;\Q)$.
In particular, $W_{12}$ is a rational homology cobordism which, in turn, implies that the twisted cobordism map has the correct functoriality.
To show that it is an isomorphism we invoke the composition law for twisted coefficients~\cite{JabukaMark_product_formulae_2008}*{Section~2.3}.
On the one hand, we observe that in the identification of \cref{eg:standardness}(iii) we have
	\begin{equation*}
	F^\infty_{V_1,\fs}\big(\tHFoo(Y,\ft)\big)\cong \tHFoo(Y,\ft)\otimes\Lambda^nH_2(\mathbb{S}_n),
	\end{equation*}
which follows from the definition of the maps induces by 1--handles~\cite{OzsvathSzabo_4mfs_triangles_2006}*{Section~4.3}.
On the other hand, \cref{T:2-handles_essential} applies to the 2--handles with~$V=H_2(\mathbb{S}_n)$ and shows that~$F^\infty_{V_2,\fs}$ maps the image of~$F^\infty_{V_1,\fs}$ isomorphically onto~$\tHFoo(N,\ft_N)$.
We can therefore conclude that we have an isomorphism $F^\infty\colon\tHFoo\Yt\rao{\cong}\tHFoo(N,\ft_N)$, which finishes the proof.
\end{proof}

\section{Intersection forms of smooth fillings}\label{ch:intersection_forms}
We already mentioned that \cref{T:intersection form bound} imposes restrictions on the possible intersection forms of smooth $4$--manifolds with fixed boundary.
We will now make the nature of these restrictions more precise.
We begin with some general remarks about non-degenerate symmetric bilinear forms over the integers. 
Let~$L$ be a free Abelian group of rank~$n$ equipped with an integer-valued symmetric bilinear form~$S$ and let~$d={\left|\det S\right|}$. 
Recall that $S$ is called \emph{non-degenerate} if~$d\neq0$ and \emph{unimodular} if~$d=1$.
We will refer to the expressions of the form~$S(x,x)$, $x\in L$, as \emph{squares} of~$S$.
We say that $S$~is \emph{semidefinite} (or simply \emph{definite} in the non-degenerate case) if all non-zero squares 
have the same sign, and \emph{indefinite} otherwise.
Furthermore, $S$~is called \emph{even} if all squares are even, and \emph{odd} otherwise.
If $S$ is non-degenerate then $L$ canonically embeds into the \emph{dual group} $L^*=\Hom_\Z(L,\Z)$ as a subgroup of index~$d$.
Consequently, we can identify~$L$ with its image in~$L^*$ and extend~$S$ to a rational-valued form on~$L^*$ as follows.
For any $\lambda\in L^*$ we have~$d\lambda\in L$ and we set
	\begin{equation*}
	S^*(\lambda,\mu) 
	= \tfrac{1}{d^2}S(d\lambda,d\mu) 
	= \tfrac1d \lambda(d\mu)
	\in \tfrac1d\Z\subset\Q.
	\end{equation*}
for any pair~$\lambda,\mu\in L^*$. 
\begin{remark}\label{R:lattice lingo}
A less intrinsic but more geometric picture emerges when we embed $L$ as a \emph{lattice} in~$\R^n$ in such a way that~$S$ corresponds to the standard inner product with the same signature as~$S$ (which is possible by Sylvester's law of inertia).
After fixing such an embedding $L\subset\R^n$ one can conveniently think of~$L^*$ as the \emph{dual lattice}
	$\Set{y\in\R^n}{x\cdot y\in\Z\;\forall x\in L}$
leading to a chain of inclusions~$L\subset L^*\subset\R^n$ and both~$S$ and~$S^*$ are given by the relevant inner product on~$\R^n$.
\end{remark}
The main purpose for introducing~$L^*$ is that it serves as a host for the \emph{characteristic covectors} of~$S$ which form the set
	\begin{equation*}
	\chi^*(S)=\Set{\kappa\in L^*}{\kappa(x)\equiv S(x,x)\mod{2}\;\forall x\in L}.
	\end{equation*}
From these we extract a numerical invariant sometimes called the \emph{shadow length}
	\begin{equation*}
	s(S)= \min\Set{|\kappa^2|}{\kappa\in\chi^*(S)}\in\Q.
	\end{equation*}
Note that in the lattice picture $s(S)$ measures the length of the shortest characteristic covector of~$S$.
For cosmetic reasons we also introduce the \emph{shadow colength}
	\begin{equation*}
	\bar s(S) = n - s(S)\in \Q.
	\end{equation*}
To the best of our knowledge these invariants first appeared implicitly in the work of Elkies~\cites{Elkies_trivial_lattice_1995,Elkies_long_shadows_1995} which was inspired by Donaldson's theorem.
We will say more about their algebraic significance after explaining the relation to \cref{T:intersection form bound}.
Now let~$Z$ be a smooth filling of a fixed $3$--manifold~$Y$ and let~$\ker(Q_Z)$ be the kernel of the intersection form on~$H_2(Z)$.
The quotient~$L_Z=H_2(Z)/\ker(Q_Z)$ is easily seen to be free Abelian of rank $b_2^+(Z)+b_2^-(Z)$ and $Q_Z$ descends to a non-degenerate form on~$L_Z$, henceforth denoted by $S_Z$, which we will refer to as the \emph{non-degenerate intersection form} of~$Z$.
Together with the observation that~$\ker(Q_Z)$ contains the image of~$H_2(Y)$ as a subgroup of full rank the universal coefficient theorem gives identifications
	\begin{align*}
	L_Z^*
	\cong &\Set{\xi\in H^2(Z)}{\scp{\xi,x}=0\;\forall x\in\ker(Q_Z)}\big/\mathrm{torsion} \\
	= &\Set{\xi\in H^2(Z)}{\text{$\xi|_Y\in H^2(Y)$ is torsion}}\big/\mathrm{torsion}.
	\end{align*}
Moreover, an inspection of the homology sequence of the pair shows that~$L_Z^*/L_Z$ injects into the torsion subgroup of~$H_1(Y)$ so that $|\det(S_Z)|$ is bounded by the order of the torsion subgroup of $H_1(Y)$.
In order to state a more algebraic reformulation of \cref{T:intersection form bound} we introduce the notation
	\begin{equation*}
	\delta(Y)= \max\Set{\delta\Yt}{\text{$\ft\in\Spinc(Y)$ torsion}}
	\end{equation*}
which gives an invariant that does not depend on any \spinc structure but only on~$Y$.
\begin{theorem}\label{T:intersection form bound_characteristic version}
Let~$Z$ be a smooth filling of~$Y$.
If~$S_Z$ is negative definite, then any characteristic covector $\kappa\in\chi^*(S_Z)$ satisfies
	\begin{equation*}
	b_2^-(Z) + \kappa^2 \leq \delta(Y).
	\end{equation*}
In other words, we have an upper bound on the shadow colength~$\bar s (S_Z)\leq  \delta(Y)$.
\end{theorem}
\begin{proof}
This is an immediate consequence of \cref{T:intersection form bound} once we understand the relationship between \spinc structures on~$Z$ and characteristic covectors of~$S_Z$.
Since this is common folklore, we shall be brief.
By the above identification of~$L_Z^*$ any~$\fs\in\Spinc(Z)$ such that~$c_1(\fs|_Y)$ is torsion gives rise to an element~$\kappa_\fs\in L_Z^*$.
Moreover, we have~$\kappa_\fs(x)=\scp{c_1(\fs),\bar x}$ for any~$\bar x\in H_2(Z)$ representing~$x\in L_Z$ which shows that~$\kappa_\fs\in\chi^*(S_Z)$.%
	\footnote{Recall that $c_1(\fs)$ reduces to $w_2(Z)$ and that~$x^2\equiv\scp{w_2(Z),x}\mod{2}$ for all $x\in H_2(Z)$.}
One readily checks that $\chi^*(S_Z)$ has a free and transitive action of~$2L_Z^*$ which can be realized by the action of~$H^2(Z)$ on~$\Spinc(Z)$. 
Hence, all characteristic covectors have the form~$\kappa_\fs$ for some~$\fs$.
It remains to show that $\kappa_\fs^2=c_1^2(\fs)$ which, in essence, follows from the definitions and Poincar\'e duality.
\end{proof}
Having identified the shadow colength as the algebraic invariant obstructed by the $\delta$--invariant, and thus by the twisted correction terms, we now take a closer look from an algebraic perspective.
We restrict our attention to a negative definite form~$S$.
An important feature is that $\bar s(S)$ a priori lies in a bounded range
	\begin{equation}\label{eq:shadow length range}
	0 \;\le\; \bar s(S) \;\le\; \rk(L).
	\end{equation}
The right inequality holds by definition with equality precisely when~$S$ is even (both conditions are equivalent to~$0\in\chi^*(S)$).
The left inequality was first proved by Elkies~\cite{Elkies_trivial_lattice_1995} for unimodular forms and was extended by Owens and Strle~\cite{OwensStrle_Elkies_generalization_2012} to the general case.
More interestingly, their results also show that the equality~$\bar s(S)=0$ hold if and only if~$S\cong \bs{I}_r=r\scp{-1}$ where~$r=\rk(L)$.
This already shows that $\bar s$ is a powerful invariant.
\cref{T:intersection form bound_characteristic version} together with~\eqref{eq:shadow length range} yields the following.
\begin{corollary}\label{T:positive delta}
If~$Y$ has a smooth, negative semidefinite filling, then $\delta(Y)\ge 0$.
\end{corollary}
Next we observe that~$\bar s$ is additive and thus invariant under addition of~$\bs I_r$.
But any negative definite form~$S$ can be written as $S\cong S_0\oplus \bs I_r$ where $S_0$~is \emph{minimal} in the sense that it has no element of square~$-1$ so that $\bar s(S)=\bar s(S_0)$.
Moreover, the number~$r$ and the isomorphism class of~$S_0$ are uniquely determined by~$S$.

\begin{corollary}\label{T:finiteness theorem_almost even}
Let $Z$ be a smooth filling of~$Y$.
Suppose that~$S_Z$ is negative definite and splits as $S_Z\cong S_0\oplus \bs I_r$ with~$S_0$ minimal and even.
Then $\rk(S_0)\le\delta(Y)$.
In particular, there is a finite list of possible forms.
\end{corollary}
Note that this includes \cref{T:finiteness theorem} as a special case for~$r=0$. 
\begin{proof}
Since $S_0$ is even, its rank agrees with~$\bar s(S_0)=\bar s(S_Z)$ and the bound follows from \cref{T:intersection form bound_characteristic version}.
Moreover, the determinant of~$S_0$ agrees up to a sign with that of~$S_Z$ which is bounded in absolute value by the order of the torsion subgroup of~$H_1(Y)$.
Since there are only finitely many isomorphism classes of definite forms with given rank and determinant, the result follows.
\end{proof}
It is an interesting question whether the assumption that~$S_0$ is even is necessary in \cref{T:finiteness theorem_almost even}.
In essence, this was already asked by Elkies~\cite{Elkies_long_shadows_1995}*{p.650}.
\begin{question}[Elkies]\label{Q:rank bound for minimal lattices}
Does an upper bound on~$\bar s(S_0)$ for a minimal (unimodular) form~$S_0$ imply an upper bound on the rank of~$S_0$?
\end{question}
As far as we know, this question is still open.
Some evidence for an affirmative answer is available in the unimodular case. 
Elkies showed that there are exactly 14~non-trivial minimal unimodular lattices with~$\bar s(L_0)\leq8$~\cite{Elkies_trivial_lattice_1995,Elkies_long_shadows_1995}; in addition, rank bounds are known for~$\bar s(L_0)\leq 24$~\cite{Gaulter_characteristic_Stuff_2007,NebeVenkov_long_shadows_2003}.
%

\section{Computations and applications}
	\label{ch:applications}

After the abstract algebraic considerations in \cref{ch:intersection_forms} we now turn to more concrete problems. 
We begin by giving a computation of the twisted correction terms of~$\Sigma_g\times S^1$ for a  surface~$\Sigma_g$ of arbitrary genus~$g$.
For~$g\ge1$ these are arguably the simplest examples of $3$--manifolds with non-standard~$\HFoo$ and as such they are not accessible to the previously available (untwisted) correction terms.

\subsection{A surface times a circle}
Recall from \cref{eg:T3 and S1xS2} that~$\ud(S^1\times S^2,\ft_0)=-\tfrac12$ and~$\ud(T^3,\ft_0) = \tfrac12$.
It turns out that this pattern continues as follows.
\begin{theorem}\label{p:allgenera}\label{T:surface times circle}
Let~$\Sigma_g$ be a closed, oriented surface of genus~$g$.
Then the unique torsion \spinc structure~$\ft_0$ on the product $\Sigma_g\times S^1$ satisfies
	\begin{equation*}
	\ud(\Sigma_g\times S^1,\ft_0) = 
	\begin{cases}
	-\tfrac 12 & \text{$g$ even} \\
	+\tfrac 12 & \text{$g$ odd}
	\end{cases}
	\end{equation*}
In other words, we have
	\begin{equation*}
	  \delta(\Sigma_g\times S^1)
	= \delta(\Sigma_g\times S^1,\ft_0) 
	= 8\left\lceil\tfrac{g}{2}\right\rceil
	\end{equation*}
where~$\lceil\cdot\rceil$ is the ceiling function.
\end{theorem}

We split the proof into two parts.
We first exhibit an explicit filling that realises the lower bound~$\delta(\Sigma_g\times S^1)$.
The second part is an inductive argument based on a computation of~$\ud(\Sigma_2\times S^1,\ft_0)$ which will occupy most of the present section.

\begin{proposition}\label{l:even}
$\Sigma_g\times S^1$ has a smooth filling~$Z_g$ with even, negative semidefinite intersection form of rank $b_2^-(Z)=8\left\lceil\tfrac{g}{2}\right\rceil$.
In particular, we have $\delta(\Sigma_g\times S^1)\geq 8\left\lceil\tfrac{g}{2}\right\rceil$.
\end{proposition}

In \cref{p:gamma} below we will also determine the intersection form of the $4$--manifold~$Z_g$ constructed below.

\begin{proof}
We first construct a 4--manifold~$Z_g'$ as the complement of a (symplectic) genus-$g$ surface of self-intersection 0 in a blow-up of $\CP$. 
We start with a configuration of $g+1$ complex curves of which $g$~are smooth generic conics in a pencil, and the remaining one is a generic line.
This configuration has $2g$ double points and four points of multiplicity $g$. One can resolve the double points in the symplectic category, hence obtaining a symplectic curve with four points of multiplicity~$2g$.
We now blow up $\CP$ at these points, and at $4g+1$ generic points of the curve. 
Taking the proper transform gives a smooth symplectic curve~$C$ of self-intersection 0 in $X = \CP\#(4g+5)\CPbar$ in the homology class 
	\begin{equation*}
	[C] = (2g+1)h-g(e_1+\dots+e_4)-(e_5+\dots+e_{4g+5}).
	\end{equation*}
The canonical divisor $K_X$ of $X$ is Poincar\'e dual to $e_1+\dots+e_{4g+5}-3h$, hence the adjunction formula reads
	\begin{equation*}
	0 = \langle K_X, C\rangle + C^2 + \chi(C) = 4g-(4g+1)-(2g+1) + 2-2g(C),
	\end{equation*}
showing that $C$ has genus $g(C)=g$.
In particular, the complement $Z_g'$ of an open, regular neighbourhood of $C$ in $X$ is a filling of~$\Sigma_g\times S^1$ and we claim that for odd~$g$ it has all the required properties.
In fact, it is negative semidefinite, since $C^2 = 0$ and $b_2^+(X) =1$; moreover, since $b_2^-(X) = 4g+5$, we have that 
	\begin{equation*}
	b_2^-(Z_g')=b_2^-(X)-1 = 4g+4.
	\end{equation*}
Finally, $[C]$ is easily seen to be characteristic in $H_2(X)$ if $g$~is odd, hence the intersection form on the complement is even: in fact, if $x\in[C]^\perp$, then $x^2 \equiv x\cdot[C] = 0$.

For odd~$g$ we can therefore take $Z_g=Z_g'$.
For even~$g$ use the following trick.
Let~$V_g$ be a cobordism from~$\Sigma_g$ to~$\Sigma_{g+1}$ obtained by attaching a $3$--dimensional 1--handle and let~$W_g=V_g\times S^1$.
Then the intersection form on~$H_2(W_g)$ is trivial and $H_1(\Sigma_g\times S^1)$ injects into~$H_1(W_g)$.
In particular, a Mayer--Vietoris argument shows that if $Z$ is a filling of~$\Sigma_g\times S^1$ with $b_2^+(Z)=0$, then $Z\cup W_g$ is any filling of~$\Sigma_{g+1}\times S^1$ with 
	$b_2^+(Z\cup W_g)=0$ and 
	$b_2^-(Z\cup W_g)=b_2^-(Z)$.
Moreover, if $Z$ has an even intersection form, then so does $Z\cup W_g$.
So for $g$~even and positive we let $Z_g=Z_{g-1}\cup W_{g-1}$.
\end{proof}

The second ingredient for our proof of \cref{p:allgenera} is the following special case.
\begin{proposition}\label{p:genus2}
The correction term of $\Sigma_2\times S^1$ with its unique torsion \spinc structure $\ft_0$ is $\ud(\Sigma_2\times S^1,\ft_0) = -\tfrac12$.
\end{proposition}

The computation is lengthy and technical and we postpone it until \cref{sss:genus2}.
We first explain how it fits into the proof of \cref{p:allgenera}.

\begin{proof}[Proof of \cref{p:allgenera}]
For brevity we write $Y_g=\Sigma_g\times S^1$ and omit the unique torsion \spinc structure from the notation.
We proceed by induction on~$g$.
As mentioned in \cref{eg:T3 and S1xS2}, the computations of~$\ud(Y_g)$ for $g=0$ or~$1$ are covered in the literature, and the case $g=2$ is obtained in \cref{p:genus2} above.

Suppose now that $g > 2$. There is a cobordism from $Y_g$ to $Y_{g-2}\# Y_2$ obtained by attaching a single 2--handle along a null-homologous knot with framing~0. This is shown in \cref{f:handleslide}: in the top picture, the dashed curve represents the attaching curve of the 2--handle, and the other curves give a surgery presentation for $Y_g$; the bottom picture is obtained from the one on top by a handleslide, and it shows that the positive boundary of the cobordism is $Y_{g-2}\#Y_2$.
\begin{figure}
\labellist
\pinlabel $\overbrace{\phantom{--------------}}^{g-2}$ at 275 251
\pinlabel $\dots$ at 275 219
\pinlabel 0 at 370 166
\pinlabel 0 at 152 166
\pinlabel $\overbrace{\phantom{--------------}}^{g-2}$ at 275 103
\pinlabel $\dots$ at 275 71
\pinlabel 0 at 370 20
\pinlabel 0 at 152 20
\endlabellist
\includegraphics[scale=.7]{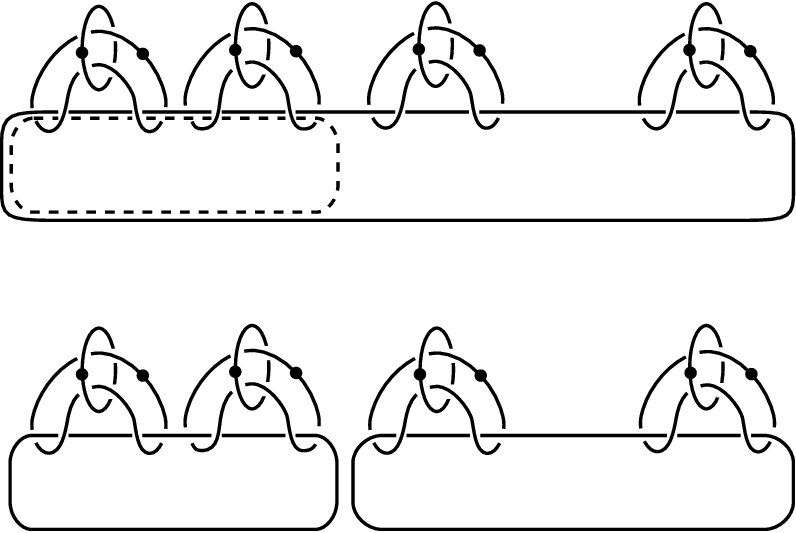}
\caption{The handleslide.}\label{f:handleslide}
\end{figure}
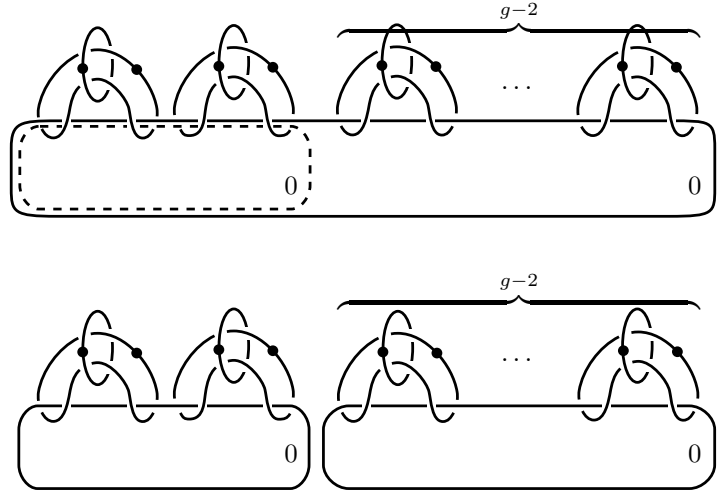
In particular, the assumptions of \cref{t:main_inequality} are satisfied by this cobordism, and applying additivity we get
	\begin{align*}
	2\ud(Y_g) + b_1(Y_g) 
	&\leq 2\ud(Y_{g-2}\# Y_2) + b_1(Y_{g-2}\# Y_2) \\
	&= 2\ud(Y_{g-2}) + 2\ud(Y_2) + b_1(Y_{g-2}\#Y_2),
	\end{align*}
showing that $\ud(Y_g)\le \ud(Y_{g-2})$.
On the other hand, \cref{l:even} ensures that $Y$ bounds an even, negative semidefinite $4$--manifold $Z$ with $b_2^-(Z) = 4g+4$ if $g$ is odd and $b_2^-(Z) = 4g$ if $g$ is even.
The fact that $Z$ is even implies that we can find a \spinc strucutre $\fs\in\Spinc(Z)$ with $c_1(\fs)$ torsion;
hence, applying \cref{T:intersection form bound} to~$\Zs$, we obtain
\[0 + b_2^-(Z) \le 4\ud(Y_g) + 2b_1(Y_g),\]
from which we get $\ud(Y_g) \ge \tfrac12$ for $g$ odd, and $\ud(Y_g) \ge -\tfrac12$ for $g$ even.
\end{proof}

\subsubsection{Computation of $\ud(\Sigma_2\times S^1,\ft_0)$}\label{sss:genus2}

This subsection is devoted to the proof of \cref{p:genus2}. In what follows, we will denote by $K$ the right-handed trefoil $T_{2,3}$, by $K^2$ the connected sum of two copies of $K$, i.e. $K^2 = T_{2,3}\#T_{2,3}$. Also, we will denote by $M(a,b,c,d)$ the manifold obtained by doing surgery along the framed link $\bf{L}$ in Figure~\ref{f:5-link}. Notice that the 0--framed component of $\bf{L}$ is distinguished, since it is the only component of Seifert genus 2 in the complement of the other components.
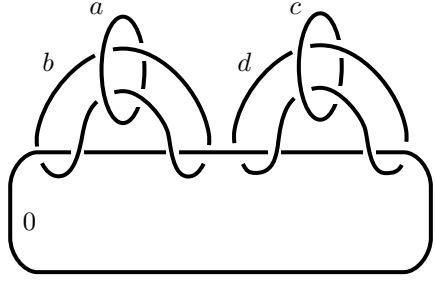
\begin{figure}
\labellist
\pinlabel $a$ at 33 100
\pinlabel $b$ at 15 80
\pinlabel $c$ at 107 100
\pinlabel $d$ at 88 80
\pinlabel $0$ at 8 20
\endlabellist
\includegraphics[scale=1]{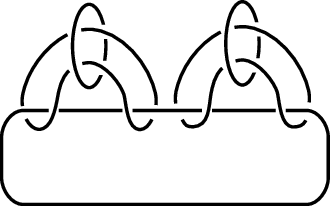}
\caption{A surgery diagram for $M(a,b,c,d)$.}\label{f:5-link}
\end{figure}
We note here the following identifications:
	\begin{equation*}
	\begin{matrix}
	M(\infty,1,1,1) \cong S^3_0(K) 
	&&
	M(0,\infty,1,1) \cong S^3_0(K)\#(S^2\times S^1) 
	\\
	M(1,1,1,1) \cong S^3_0(K^2)
	&&
	M(0,0,0,\infty) \cong T^3\#(S^2\times S^1)
	\\
	M(0,0,\infty,1) \cong T^3 
	&&
	M(0,0,0,0) \cong \Sigma_2\times S^1
	\end{matrix}
	\end{equation*}

When a $3$--manifold admits a unique torsion \spinc structure (and this is the case for all manifolds in this section, except for one, in the proof of \cref{l:doubletrefoil0}), we suppress the \spinc structure from the notation.

\begin{remark}\label{r:kunneth}
Note that the connected sum formula for Heegaard Floer homology with twisted coefficients implies that taking a connected sum with a (twisted coefficients) L-space $\Yt$ (i.e. $\tHFpred\Yt = 0$) corresponds to a degree-shift by $\ud\Yt$; in particular, since $S^2\times S^1$ is a twisted coefficients L-space with correction term~$-\tfrac12$, the groups $\tHFp(M(0,0,\infty,1))$ and $\tHFp(M(0,0,0,\infty))$ are easily computed from the corresponding groups $\tHFp(S^3_0(K_1))$ and $\tHFp(T^3)$, respectively.
These two latter groups have in fact been computed in~\cite[Lemma 8.6, Proposition 8.5]{OzsvathSzabo_4mfs_gradings_2003}.
\end{remark}

In what follows, we denote by $\bL(s)_d$ the ring $\bF[s,s\inv]$ of Laurent polynomials in the variable $s$ over the field $\bF$, supported in degree $d$; we denote by $\bL(s,t)_d$ the ring $\bF[s,s\inv,t,t\inv]$, supported in degree $d$. 
More generally, given a module $M$ over a ring~$R$ it will be convenient to write $M(s)$ for the module $M\otimes_R R[s,s\inv]$.
Also, given an element $r\in R$, we denote by $\pi_r$  the projection $M \to M/(r-1)M$.

\begin{lemma}\label{l:trefoil0}
Identify $\bF[H_2(S^3_0(K))]$ with $\bL(s)$. The plus-hat long exact sequence for the twisted Heegaard Floer homology of $S^3_0(K)$ reads:
\[
\xymatrix{
\tHFp(S^3_0(K))\ar[r]\ar^{=}[d] & \tHFhat(S^3_0(K))\ar[r]\ar^{=}[d] & \tHFp(S^3_0(K))\ar^{=}[d] \\
\Tp_{-\frac12}\oplus \bL(s)_{-\frac32}\ar[r]^-{\left(\begin{smallmatrix}0 & 1-s\\ 0 & 0\end{smallmatrix}\right)} &
\bL(s)_{-\frac12}\oplus \bL(s)_{-\frac32} \ar[r]^-{\left(\begin{smallmatrix}\pi_{s} & 0\\ 0 & 1\end{smallmatrix}\right)} &
\Tp_{-\frac12}\oplus \bL(s)_{-\frac32}
}
\]
\end{lemma}

\begin{proof}
Recall that $K$ is an L-space knot, and that in fact $S^3_1(K)$ is an L-space with $d(S^3_1(K)) = -2$. It follows that $\HFp(S^3_1(K)) = \tHFp(S^3_1(K)) = \Tp_{-2}$ and $\HFhat(S^3_1(K)) = \tHFhat(S^3_1(K)) = \bF_{-2}$.

Consider the surgery exact triangle associated to 0--surgery along $K$; since the hat-plus long exact sequence is natural with respect to cobordisms, this triangle fits into a long exact sequence of triangles as follows:
\[
\xymatrix{
& \vdots\ar[d] & \vdots\ar[d] & \vdots\ar[d] & \\
\dots\ar[r] & \tHFp(S^3)(s)\ar[r]\ar[d] & \tHFhat(S^3)(s)\ar[r]\ar[d] & \tHFp(S^3)(s)\ar[d]\ar[r] & \dots\\
\dots\ar[r] & \tHFp(S^3_0(K))\ar[r]\ar[d] & \tHFhat(S^3_0(K))\ar[r]\ar[d] & \tHFp(S^3_0(K))\ar[d]\ar[r] & \dots\\
\dots\ar[r] & \tHFp(S^3_1(K))(s)\ar[r]\ar[d] & \tHFhat(S^3_1(K))(s)\ar[r]\ar[d] & \tHFp(S^3_1(K))(s)\ar[r]\ar[d] & \dots\\
& \vdots & \vdots & \vdots &
}
\]

Observe that the horizontal maps from plus to hat have degree $+1$, the next ones have degree $0$ and the following ones have degree $-2$; that the vertical maps from $S^3_1(K)$ to $S^3$ are a sum of maps of non-negative degree, while all the other ones involving the torsion \spinc structure on $S^3_0(K)$ have degree $-\tfrac12$.
Finally, since $\tHFoo(S^3_0(K)) = \Too$, we also obtain that $\tHFp(S^3_0(K))$ contains a single tower. It follows that the vertical map $\tHFp(S^3_1(K))(s)\to\tHFp(S^3)$ is (up to a sign) multiplication by $U(1-t)$. An easy diagram chase completes the proof.
\end{proof}

\begin{lemma}\label{l:doubletrefoil0}
Identify $\bF[H_2(S^3_0(K^2))]$ with $\bL = \bL(s)$. The plus-hat long exact sequence for the twisted Heegaard Floer homology of $S^3_0(K^2)$ in the torsion \spinc structure reads:
\[
\xymatrix{
\tHFp(S^3_0(K^2))\ar[r]\ar^{=}[d] & \tHFhat(S^3_0(K^2))\ar[r]\ar^{=}[d] & \tHFp(S^3_0(K^2))\ar^{=}[d] \\
\Tp_{-\frac12}\oplus\bL_{-\frac32}\oplus\bL_{-\frac52}\ar[r]^-{\left(\begin{smallmatrix}0 & 1-s & 0\\ 0 & 0 & 1 \\ 0 & 0 & 0\\ 0 & 0 & 0\end{smallmatrix}\right)} &
\bL_{-\frac12}\oplus\bL^2_{-\frac32}\oplus\bL_{-\frac52}\ar[r]^-{\left(\begin{smallmatrix}\pi_s & 0 & 0 & 0\\ 0 & 0 & 1 & 0\\ 0 & 0 & 0 & 1\end{smallmatrix}\right)} &
\Tp_{-\frac12}\oplus\bL_{-\frac32}\oplus\bL_{-\frac52}
}
\]
\end{lemma}

\begin{proof}
This is analogous to the proof of \cref{l:trefoil0}, hence here we only outline the differences. $K^2$ is not an L-space knot, but the Heegaard Floer homology of $S^3_{12}(K^2)$ was computed in~\cite[Lemma 4.1]{OzsvathSzabo_plumbed_2003}, at least in the \spinc structure which is relevant for the computation of $\HFp(S^3_0(K^2))$, and which is relevant to us (called  $Q(0)$ in loc.\ cit.). Namely:
\[
\HFp(S^3_{12}(K^2),Q(0)) = \bF_{-\frac34}\oplus\Tp_{-\frac34}
\]
Instead of using the surgery exact triangle for $+1$--surgery, we use the triangle for twisted $+12$--surgery, where the degrees of the vertical maps are $-\tfrac94$ (from $0$--surgery to $12$--surgery), $-\tfrac{11}4$ (from $12$--surgery to $S^3$) and $-\tfrac12$ (from $S^3$ to the 0--surgery).
A diagram chase as above proves the lemma.
\end{proof}

\begin{lemma}\label{l:M0111}
$\tHFpred(M(0,1,1,1))$ is supported in degrees at most $-2$, and 
	\begin{equation*}
	\ud\big(M(0,1,1,1)\big) = -1.
	\end{equation*}
\end{lemma}

\begin{proof}
We need to set up some notation. Let $Y=M(0,1,1,1)$; $H_2(Y)$ is generated by classes $s$ and $t$, where $s$ is represented by a capped-off Seifert surface for the marked component of $\bf{L}$, and $t$ is represented by a capped-off Seifert surface for the first component of $\bf{L}$ (i.e. the one with framing 0 in this surgery). This identifies $\bF[H_2(Y)]$ with $\bL(s,t)$.

Since $Y=M(0,1,1,1)$ fits into an surgery triangle with $S^3_0(K)=M(\infty,1,1,1)$ and $S^3_0(K^2)=M(1,1,1,1)$, we have the following long exact sequence of exact triangles, as in the proof of \cref{l:trefoil0}:
\[
\xymatrix{
& \vdots\ar[d] & \vdots\ar[d] & \vdots\ar[d] & \\
\dots\ar[r] & \tHFp(S^3_0(K))(t)\ar[r]^\alpha\ar[d] & \tHFhat(S^3_0(K))(t)\ar[r]^\beta\ar[d] & \tHFp(S^3_0(K))(t)\ar[d]\ar[r] & \dots\\
\dots\ar[r] & \tHFp(Y)\ar[r]\ar[d] & \tHFhat(Y)\ar[r]\ar[d] & \tHFp(Y)\ar[d]\ar[r] & \dots\\
\dots\ar[r] & \tHFp(S^3_0(K^2))(t)\ar[r]^{\alpha^2}\ar[d]^-{F^+} & \tHFhat(S^3_0(K^2))(t)\ar[r]^{\beta^2}\ar[d]^-{\hat F} & \tHFp(S^3_0(K^2))(t)\ar[r]\ar[d]^-{F^+} & \dots\\
& \vdots & \vdots & \vdots &
}
\]
Notice that the ``new'' variable $t$ is the one associated with the \emph{second} Seifert surface, since the first Seifert surface generates the homology of $S^3_0(K)$ and $S^3_0(K^2)$.

Notice that, since $\tHFpred(S^3_0(K))$ is supported in degree $-\tfrac32$ and the map from $\tHFp(S^3_0(K))$ to $\tHFp(Y)$ has degree $-\tfrac12$, the image of $\tHFp(S^3_0(K))(t)$ in $\tHFpred(Y)$ is supported in degrees at most $-2$. In order to prove the statement, it is therefore enough to prove that the image of $\tHFpred(Y)$ in $\tHFp(S^3_0(K^2))(t)$ is supported in degrees at most $-\tfrac52$ since the map $\tHFp(Y)\to\tHFp(S^3_0(K^2))(t)$, too, has degree~$-\tfrac12$.
This is in turn equivalent to showing that the vertical map starting from $\bL(s,t)_{-\frac32}\subset\tHFp(S^3_0(K^2))$ is nonzero. Let $x_0 := 1\in\bL(s,t)_{-\frac32}$.

For degree reasons, the image $F^+(x_0)$ of $x_0$ in $\tHFp(S^3_0(K))$ lies in the reduced part, which is a copy of $\bL(s,t)$. Hence, it is torsion if and only if it vanishes; our assumption becomes that $F^+(x_0) = 0$.

Observe that the map $F^+$ restricts to multiplication by $\pm1$ on the tower $\Tp(t)$, as in the proof of \cref{l:trefoil0}. Since the bottom-most element of the tower is in the image of $\beta^2$, by commutativity of the diagram, it follows that the restriction of $\hat F$ to the subspace $\bL(s,t)_{-\frac12}$ of $\tHFhat(S^3_0(K^2))$ does not vanish.

For the same reason, since $\alpha^2(x_0)$ lies in the same subspace, we obtain that $\alpha(F^+(x_0))$ does not vanish either, hence $F^+(x_0)\neq 0$, as required.

Finally, notice that the image of $F^+$ cannot contain the bottom-most element of the tower of $\tHFp(S^3_0(K))$: the restriction of $F^+$ onto the tower of $\tHFp(S^3_0(K^2))$ certainly does not, and the reduced part is supported in 
lower degrees.
\end{proof}

\begin{proof}[Proof of \cref{p:genus2}]
The argument here will be similar to the one seen in the proofs of the three lemmas above.
We first claim that $\tHFpred(M(0,0,1,1))$ is supported in degrees at most $-\tfrac32$ and that $\ud(M(0,0,1,1)) = -\tfrac32$: in fact, $M(0,0,1,1)$ fits into a surgery triple with $M(0,1,1,1)$ and $M(0,\infty,1,1)\cong S^3_0(K)\#(S^2\times S^1)$, which gives the following long exact sequence in Heegaard Floer homology:
\[
\dots\to\tHFp(M(0,\infty,1,1))(t)\to\tHFp(M(0,0,1,1))\to\tHFp(M(0,1,1,1))(t)\to\dots
\]
Here $t$ is the new generator corresponding to the new class, and needs not be confused with the $t$ used above.

However, combining \cref{l:trefoil0} and \cref{r:kunneth} we obtain that the reduced part of the leftmost group is supported in degrees at most $-2$ and the same holds for the rightmost group, thanks to \cref{l:M0111}. The same argument as above shows that $\tHFpred(M(0,0,1,1))$ is supported in degrees at most $-\tfrac32$ and allows for the computation of $\ud$.

An analogous computation, combined with~\cite[Proposition 8.5]{OzsvathSzabo_4mfs_gradings_2003}, shows that $\tHFpred(M(0,0,0,1))$ is supported in negative degrees and that $\ud(M(0,0,0,1)) = 0$; this time, however, the map from the tower of $\tHFp(M(0,0,1,1))(t)$ is no longer injective, so one needs to be more careful. Let us now look at the surgery triple involving $M(0,0,0,0)$, $M(0,0,0,1)$ and $M(0,0,0,\infty)$: we have
\[
\dots\to\tHFp(M(0,0,0,\infty))(t)\to\tHFp(M(0,0,0,0))\to\tHFp(M(0,0,0,1))(t)\to\dots
\]
and, as above, using \cref{r:kunneth} and the computation of $\tHFp(T^3)$, we conclude the proof of the proposition.
\end{proof}

\subsubsection{Intersection forms of fillings of $\Sigma_g\times S^1$}
With \cref{T:intersection form bound_characteristic version,T:surface times circle} at our disposal, we have a concrete restriction for intersection forms of smooth fillings of~$\Sigma_g\times S^1$ on which we now elaborate.
Note that since $H_1(\Sigma_g\times S^1)$ is torsion-free, the non-degenerate intersection form~$S_Z$ of any filling~$Z$ is unimodular. 
Furthermore, if~$Z$ is smooth and~$S_Z$ is negative definite, which we henceforth assume, then its shadow colength is bounded by $\bar s(S_Z)\le 8\left\lceil\tfrac{g}{2}\right\rceil$.
%
As before we write~$S_Z$ as~$S_0\oplus\bs I_r$ with $S_0$ minimal.
What are the possibilities for~$S_0$?
Obviously, the trivial form is realised by~$\Sigma_g\times D^2$.
In the proof of \cref{p:allgenera} we constructed a filling~$Z_g$ with a more interesting intersection form which we now determine.
\begin{example}\label{eg: Gamma fillings}
Recall that the vectors of the form $e_i+e_j$ and $\tfrac12(e_1+\dots+e_n)$ in~$\R^n$ generate a lattice~$\Gamma_{n}\subset\R^{n}$ which is 
	unimodular for~$n=4k$, 
	even for~$n=8k$,
	and odd for~$n=8k+4$.
Moreover, one can show that $\Gamma_{4k}$ is irreducible (hence minimal) and satisfies
	$\bar s(\Gamma_{8k})=\bar s(\Gamma_{8k+4})=8k$.
Recall that~$Z_g$ was obtained as the complement of a genus~$g$ surface in $\CP\#(4g+5)\CPbar$ in the homology class 
	$x=(2g+1)h-g(e_1+\dots+e_4)-(e_5+\dots+e_{4g+5})$.
By \cref{p:gamma} below $S_{Z_g}$ is isomorphic to~$\Gamma_{4g+4}$.
Moreover, for any~$h<g$ we can extend~$Z_h$ to a filling of~$Z_g$ by adding a cobordism as in the proof of \cref{p:allgenera}. 
By blowing up these fillings we can realise the forms~$\Gamma_{4h+4}\oplus\bs I_r$ with $h\le g$ and~$r\ge 0$ by smooth fillings of~$\Sigma_{g}\times S^1$. 
\end{example}
\begin{lemma}\label{p:gamma}
Fix a positive integer $g$ (not necessarily odd), and let $\langle x\rangle$ be the the sublattice of $H_2(\CP\#(4g+5)\CPbar)$ generated by the class:
\[x=(2g+1)h-g(e_1+e_2+e_3+e_4)-(e_5+\dots+e_{4g+5}).\]
The lattice $Q = \langle x\rangle^\perp/\langle x\rangle$ is of type $\Gamma_{4g+4}$.
\end{lemma}
\begin{proof}
For brevity we let $n=4g+4$.
Since $Q$ is the intersection form of a $4$--manifold whose boundary has torsion-free homology, it is unimodular (and in particular integral). Moreover, $Q$ is a \emph{root lattice}, since the associated vector space is generated by the set $\mathcal{R}$ of \emph{roots} (i.e. elements of square $-2$)
	\begin{equation*}
	\mathcal{R} = 
		\{e_1-e_2, \,
		  e_2-e_3, \,
		  e_3-e_4, \,
		  e_5-e_6,
		  \dots, \,
		  e_{n}-e_{n+1}, \, 
		  h-e_1-e_2-e_5\}.
	\end{equation*}

\[
\xygraph{
!{<0cm,0cm>;<1cm,0cm>:<0cm,1cm>::}
!{(0,2)}*+{\bullet}="a"
!{(0,0)}*+{\bullet}="b"
!{(1.6,1) }*+{\bullet}="d"
!{(3.2,1) }*+{\bullet}="e"
!{(4.8,1) }*+{\bullet}="h"
!{(7.2,1) }*+{\bullet}="f"
!{(8.8,1) }*+{\bullet}="g"
!{(1.6,1.5) }*+{e_2-e_3}
!{(3.2,0.6) }*+{h-e_1-e_2-e_5}
!{(4.8,1.5) }*+{e_5-e_6}
!{(0,2.4) }*+{e_1-e_2}
!{(6,1)}*+{\cdots}
!{(0,-0.4) }*+{e_3-e_4}
!{(7.2,0.6) }*+{e_{n-1}-e_{n}}
!{(8.8,1.5) }*+{e_{n}-e_{n+1}}
"d"-"e"
"e"-"h"
"a"-"d"
"d"-"b"
"f"-"g"
}
\]

The roots in $\mathcal{R}$ intersect according to the Dynkin diagram~$D_n$ shown above; therefore, $\mathcal{R}$ generates a sublattice $Q'\subset Q$ of index~2 isomorphic to the lattice $D_n$. We think of $D_n$ as sitting in $\R^n$ (with orthonormal basis $\{f_1,\dots,f_n\}$), generated by the roots $\{f_1+f_2, f_1-f_2,\dots,f_{n-1}-f_n\}$.

In particular, the two ``short legs'' of the Dynkin diagram are $f_1+f_2$ and $f_1-f_2$. Recall that there are only two unimodular overlattices of $D_n$ up to isomorphism: $\Z^n$ and $\Gamma_n$, both sitting in $\R^n$; see~\cite[Section 1.4]{Ebeling_latticesandcodes_3rd_2013}. 
The overlattice is $\Z^n$ if and only if it contains $f_1$, i.e. only if it contains half of the sum of the two ``short legs'' of the Dynkin diagram, and it is $\Gamma_n$ otherwise. Since the action of the Weyl group on the set of fundamental sets of roots is transitive, we may assume that the isomorphism between $Q'$ and $D_n$ identifies the two chosen bases. The two short legs of $\mathcal{R}$ are $e_1-e_2$ and $e_3-e_4$, and their sum $y$ is represented by vectors $y + X$ none of which divisible by 2 in $Q$: if $k$ is odd $\langle h,y+kx\rangle$ is odd, and if $k$ is even $\langle e_1, y+kh\rangle$ is odd. Hence $Q$ is isomorphic to $\Gamma_n$.
\end{proof}

For fillings of~$T^3$ and~$\Sigma_2\times S^1$ we have~$\bar{s}(S_Z)=\bar{s}(S_0)\leq 8$.
As mentioned at the end of \cref{ch:intersection_forms}, this leaves 14 possibilities for~$S_0$, assuming that it is non-trivial (see Elkies' list in~\cite{Elkies_trivial_lattice_1995}*{p.326}). 
Among these we find~$\Gamma_8\cong E_8$ and~$\Gamma_{12}$ from \cref{eg: Gamma fillings} above.
It is well-known that~$E_8$ is the only non-trivial even lattice of rank at most~8.
As an immediate consequence we get a slightly stronger version of \cref{T:fillings of T3}.

\begin{corollary}\label{T:fillings of T3_almost even}
Let~$Z$ be a smooth filling of~$T^3$ or~$\Sigma_2\times S^1$ with~$S_Z$ negative definite of the form~$S_0\oplus \bs I_r$ with~$S_0$ even.
Then~$S_0$ is either trivial or isomorphic to~$E_8$ and both cases occur.
\end{corollary}

All other lattices in Elkies' list are odd of rank at least~12.
We have seen in \cref{eg: Gamma fillings} that $\Gamma_{12}$ is realised by a smooth filling of~$\Sigma_2\times S^1$.
However, we do not know whether it also appears for~$T^3$; \cref{T:intersection form bound_characteristic version} does not provide any obstruction in this case.
We therefore ask:

\begin{question}
Can $\Gamma_{12}$ be realised by smooth fillings of~$T^3$?
\end{question}

Of course, there is the more general question which odd lattices in Elkies' list, if any, appear for~$T^3$.
This should be compared to Fr{\o}yshov's work on fillings of the Poincaré sphere~\cite{Froyshov_SW_boundary_1996}*{Proposition~2}.

Our findings about $T^3$ and $\Sigma_2\times S^1$ leave the possibility that \cref{T:intersection form bound_characteristic version} is the only obstruction for realising even lattices.
In order to test this we consider $\Sigma_3\times S^1$ and $\Sigma_4\times S^1$ for which $\bar{s}(S_Z)\leq 16$.
The only non-trivial even lattices of rank at most~16 are~$E_8$, $E_8\oplus E_8$, and~$\Gamma_{16}$ (see~\cite{ConwaySloane_spherepackings_3rd_1999}*{Table~16.7}).
In \cref{eg: Gamma fillings} we have realised all but~$E_8\oplus E_8$.

\begin{question}
Can $E_8\oplus E_8$ be realised by a smooth filling of $\Sigma_3\times S^1$ or $\Sigma_4\times S^1$?
\end{question}

In the same spirit one may wonder which of the 24 even lattices of rank 24 (see \cite{ConwaySloane_spherepackings_3rd_1999}*{Ch.~18} for a list) are realised by fillings of~$\Sigma_5\times S^1$ or $\Sigma_6$.
For example:

\begin{question}
Can the Leech lattice be realised by a smooth filling of $\Sigma_5\times S^1$?
\end{question}

Another way to look at this question is which is the simplest $3$--manifold that can appear on the boundary of a smooth $4$--manifold whose non-degenerate intersection form is a given lattice, for example the Leech lattice. 
In this special case $\Sigma_5\times S^1$ might be considered a satisfactory answer.

If we increase the genus further, thereby further weakening the bound on~$\bar{s}(S_Z)$, we soon enter uncharted algebraic territory. 
As mention above, up to shadow colength~24 there is a manageable list of even unimodular lattices and the number of minimal lattices is known to be finite but possibly large~\cite{Gaulter_characteristic_Stuff_2007}*{Ch.~5}. 
Beyond this range the number of lattices allowed by \cref{T:intersection form bound_characteristic version} explodes, rendering any attempt of an enumeration extremely difficult if not impossible. 
As a consequence, without further obstructions there is little hope for a classification of all lattices that can appear as non-degenerate intersection forms of fillings of $3$--manifolds with large $\delta$--invariants.

\subsection{Embeddings into closed $4$--manifolds}

In this section, $P$ will be a fixed integral homology sphere such that $d := d(P,\ft_0) \neq 0$ for the unique \spinc structure $\ft_0$ on $P$. In what follows, the \spinc structure will be omitted from the notation whenever possible; also, $nP$ will denote the connected sum of $|n|$ copies of $P$ or $-P$ (i.e. $P$ with the reversed orientation) depending on the sign of~$n$.
For example, the Poincar\'e sphere satisfies these requirements, since $d(S^3_{+1}(T_{2,3})) = -2$.

\begin{proposition}
Fix a $3$--manifold $Y$ and a closed, smooth $4$--manifold $X$ with definite intersection form. 
There exists an integer~$N$, depending only on~$Y$, such that $Y_n=Y\#nP$ does not embed in~$X$ as a separating hypersurface for~$|n| > N$. 
\end{proposition}

\begin{proof}
Notice that the statement is independent on the orientation of $P$, thus we can pick the orientation of $P$ for which $d>0$. In particular, as $P$ is an integral homology sphere, $d$ is an even integer, hence $d\ge 2$. Similarly, we can assume that~$X$ is negative definite.

For each $n$ there is an isomorphism $\Spinc(Y) \to \Spinc(Y_n)$ defined by $\ft \mapsto \ft\#\ft_n$ that carries torsion \spinc structures to torsion \spinc structures, where $\ft_n$ is the unique \spinc structure on $nP$.

Now assume that there is a separating embedding $Y_n\hra X$. 
Let $Z$ and~$Z'$ be the closures of the connected components of $X\setminus Y_n$, labelled so that $\del Z = -\del Z' = Y_n$.
Notice that both $Z$ and~$Z'$ are negative semidefinite so that 
\cref{T:positive delta}
yields $\delta(Y_n) \ge 0$ and also $\delta(-Y_n)\geq 0$.
Since correction terms are additive, so is $\delta$, hence
\[
0\le \delta(Y_n) = \delta(Y)+4nd
\eqand
0\le \delta(-Y_n) = \delta(-Y)-4nd 
\]
It follows that $-\delta(Y)/4d\leq n\le \delta(-Y)/4d$.
We can now choose
\[
N = \left\lfloor\max\left\{\frac{\delta(Y)}8, \frac{\delta(-Y)}8\right\}\right\rfloor.\qedhere
\]
\end{proof}

Note that if $\delta(Y)$ and $\delta(-Y)$ are both negative, then $N$ can be chosen to be $-1$, hence $Y_n$ never embeds in a closed definite $4$--manifold.

\begin{example}
Let $Y=\Sigma_g\times S^1$ 
and $P$ the Poincar\'e homology sphere; since $\delta(Y) = 8\lceil \tfrac{g}2\rceil$ and $d(P) = -2$, we get that $Y\#nP$ does not embed in a negative definite $4$--manifold as a separating hypersurface if $n>\lceil \tfrac{g}2 \rceil$.
In particular, when $Y$ is either the $3$--torus $T^3$ or $\Sigma_2\times S^1$, this shows that $Y\#nP$ cannot be embedded in a negative definite $4$--manifold $X$ whenever $n\ge 2$.
\end{example}

\bibliography{pack/correctionterms}

\begin{bibdiv}
\begin{biblist}

\bib{Brown_cohomology_of_groups_1982}{book}{
      author={Brown, Kenneth~S.},
       title={Cohomology of groups},
      series={Graduate Texts in Mathematics},
   publisher={Springer Verlag, New York-Berlin},
        date={1982},
      volume={87},
        ISBN={0-387-90688-6},
}

\bib{ConwaySloane_spherepackings_3rd_1999}{book}{
      author={Conway, J.~H.},
      author={Sloane, N. J.~A.},
       title={Sphere packings, lattices and groups},
     edition={Third},
      series={Grundlehren der Mathematischen Wissenschaften},
   publisher={Springer Verlag, New York},
        date={1999},
      volume={290},
        ISBN={0-387-98585-9},
         url={http://dx.doi.org/10.1007/978-1-4757-6568-7},
}

\bib{Ebeling_latticesandcodes_3rd_2013}{book}{
      author={Ebeling, Wolfgang},
       title={Lattices and codes},
      series={Advanced {L}ectures in {M}athematics},
   publisher={Springer Verlag},
        date={2013},
}

\bib{Elkies_trivial_lattice_1995}{article}{
      author={Elkies, Noam~D.},
       title={A characterization of the {${\mathbb{Z}}^n$} lattice},
        date={1995},
        ISSN={1073-2780},
     journal={Math. Res. Lett.},
      volume={2},
      number={3},
       pages={321\ndash 326},
         url={http://dx.doi.org/10.4310/MRL.1995.v2.n3.a9},
}

\bib{Elkies_long_shadows_1995}{article}{
      author={Elkies, Noam~D.},
       title={Lattices and codes with long shadows},
        date={1995},
     journal={Math. Res. Lett.},
      volume={2},
      number={5},
       pages={643\ndash 651},
}

\bib{Froyshov_SW_boundary_1996}{article}{
      author={Fr{\o}yshov, Kim~A.},
       title={The {S}eiberg-{W}itten equations and four-manifolds with
  boundary},
        date={1996},
        ISSN={1073-2780},
     journal={Math. Res. Lett.},
      volume={3},
      number={3},
       pages={373\ndash 390},
         url={http://dx.doi.org/10.4310/MRL.1996.v3.n3.a7},
}

\bib{Froyshov_monopolehomology_2010}{article}{
      author={Fr{\o}yshov, Kim~A.},
       title={Monopole {F}loer homology for rational homology 3-spheres},
        date={2010},
     journal={Duke Math. J.},
      volume={155},
      number={3},
       pages={519\ndash 576},
}

\bib{Gaulter_characteristic_Stuff_2007}{article}{
      author={Gaulter, Mark},
       title={Characteristic vectors of unimodular lattices which represent
  two},
        date={2007},
        ISSN={1246-7405},
     journal={J. Th\'eor. Nombres Bordeaux},
      volume={19},
      number={2},
       pages={405\ndash 414},
         url={http://jtnb.cedram.org/item?id=JTNB_2007__19_2_405_0},
}

\bib{JabukaMark_product_formulae_2008}{article}{
      author={Jabuka, Stanislav},
      author={Mark, Thomas~E.},
       title={Product formulae for {O}zsv\'ath-{S}zab\'o 4-manifold
  invariants},
        date={2008},
        ISSN={1465-3060},
     journal={Geom. Topol.},
      volume={12},
      number={3},
       pages={1557\ndash 1651},
         url={http://dx.doi.org/10.2140/gt.2008.12.1557},
}

\bib{LevineRuberman_correctionterms_arxiv}{inproceedings}{
      author={Levine, Adam~Simon},
      author={Ruberman, Daniel},
       title={Generalized {H}eegaard {F}loer correction terms},
        date={2014},
   booktitle={Proceedings of the {G}\"okova {G}eometry-{T}opology {C}onference
  2013},
   publisher={G\"okova Geometry/Topology Conference (GGT), G\"okova},
       pages={76\ndash 96},
}

\bib{LevineRubermanStrle_nonorientable_arxiv}{article}{
      author={Levine, Adam~Simon},
      author={Ruberman, Daniel},
      author={Strle, Sa{\v{s}}o},
       title={Nonorientable surfaces in homology cobordisms},
        date={2015},
        ISSN={1465-3060},
     journal={Geom. Topol.},
      volume={19},
      number={1},
       pages={439\ndash 494},
         url={http://dx.doi.org/10.2140/gt.2015.19.439},
        note={With an appendix by Ira M. Gessel},
}

\bib{NebeVenkov_long_shadows_2003}{article}{
      author={Nebe, Gabriele},
      author={Venkov, Boris},
       title={Unimodular lattices with long shadow},
        date={2003},
        ISSN={0022-314X},
     journal={J. Number Theory},
      volume={99},
      number={2},
       pages={307\ndash 317},
         url={http://dx.doi.org/10.1016/S0022-314X(02)00079-3},
}

\bib{OwensStrle_Elkies_generalization_2012}{article}{
      author={Owens, Brendan},
      author={Strle, Sa{\v{s}}o},
       title={A characterization of the {$\Bbb Z^n\oplus\Bbb Z(\delta)$}
  lattice and definite nonunimodular intersection forms},
        date={2012},
        ISSN={0002-9327},
     journal={Amer. J. Math.},
      volume={134},
      number={4},
       pages={891\ndash 913},
         url={http://dx.doi.org/10.1353/ajm.2012.0026},
}

\bib{OzsvathSzabo_4mfs_gradings_2003}{article}{
      author={Ozsv{\'a}th, Peter~Steven},
      author={Szab{\'o}, Zolt{\'a}n},
       title={Absolutely graded {F}loer homologies and intersection forms for
  four-manifolds with boundary},
        date={2003},
        ISSN={0001-8708},
     journal={Adv. Math.},
      volume={173},
      number={2},
       pages={179\ndash 261},
         url={http://dx.doi.org/10.1016/S0001-8708(02)00030-0},
}

\bib{OzsvathSzabo_plumbed_2003}{article}{
      author={Ozsv{\'a}th, Peter~Steven},
      author={Szab{\'o}, Zolt{\'a}n},
       title={On the {F}loer homology of plumbed three-manifolds},
        date={2003},
        ISSN={1465-3060},
     journal={Geom. Topol.},
      volume={7},
       pages={185\ndash 224},
         url={http://dx.doi.org/10.2140/gt.2003.7.185},
}

\bib{OzsvathSzabo_HF_properties_2004}{article}{
      author={Ozsv{\'a}th, Peter~Steven},
      author={Szab{\'o}, Zolt{\'a}n},
       title={Holomorphic disks and three-manifold invariants: properties and
  applications},
        date={2004},
        ISSN={0003-486X},
     journal={Ann. of Math. (2)},
      volume={159},
      number={3},
       pages={1159\ndash 1245},
         url={http://dx.doi.org/10.4007/annals.2004.159.1159},
}

\bib{OzsvathSzabo_HF_definition_2004}{article}{
      author={Ozsv{\'a}th, Peter~Steven},
      author={Szab{\'o}, Zolt{\'a}n},
       title={Holomorphic disks and topological invariants for closed
  three-manifolds},
        date={2004},
        ISSN={0003-486X},
     journal={Ann. of Math. (2)},
      volume={159},
      number={3},
       pages={1027\ndash 1158},
         url={http://dx.doi.org/10.4007/annals.2004.159.1027},
}

\bib{OzsvathSzabo_4mfs_triangles_2006}{article}{
      author={Ozsv{\'a}th, Peter~Steven},
      author={Szab{\'o}, Zolt{\'a}n},
       title={Holomorphic triangles and invariants for smooth four-manifolds},
        date={2006},
        ISSN={0001-8708},
     journal={Adv. Math.},
      volume={202},
      number={2},
       pages={326\ndash 400},
         url={http://dx.doi.org/10.1016/j.aim.2005.03.014},
}

\end{biblist}
\end{bibdiv}


\end{document}